\documentclass[11pt,reqno]{amsart}

\usepackage{hyperref}
\hypersetup{
	colorlinks=true,
	citecolor=blue,
	linkcolor=blue,
	filecolor=magenta,      
	urlcolor=cyan,
}
\usepackage{fullpage}
\usepackage{amsmath, amsthm, amssymb, graphicx, dsfont, stmaryrd, extarrows}
\usepackage[T1]{fontenc}
\usepackage{textcomp}
\usepackage{lmodern}
\usepackage{microtype} 

\usepackage[color,matrix,arrow,all]{xy}
\providecommand{\1}{}
\renewcommand{\1}{\mathds{1}}

\usepackage{amscd}
\usepackage{mathrsfs}
\usepackage{tikz-cd}
\usetikzlibrary{matrix,patterns}
\usepackage[colorinlistoftodos]{todonotes}
\usepackage{color}
\usepackage{bbm}
\usepackage{verbatim}
\usepackage[mathscr]{euscript}

\theoremstyle{plain}
\newtheorem{thm}{Theorem}[section]
\newtheorem*{thm*}{Theorem}

\newtheorem{prop}[thm]{Proposition}
    
\newtheorem{cor}[thm]{Corollary}       
\newtheorem{lem}[thm]{Lemma}

\theoremstyle{definition} 
\newtheorem{defn}[thm]{Definition} 
\newtheorem{hyp}[thm]{Hypothesis}
\newtheorem{ex}[thm]{Example}
\newtheorem{exs}[thm]{Examples}
\newtheorem{rem}[thm]{Remark}   

\newtheorem{nota}[thm]{Notation}

\newcommand{\A}{\mathscr{A}}
\newcommand{\C}{\mathscr{C}}

\newcommand{\Dsf}{\mathsf{D}}

\newcommand{\G}{\mathcal{G}}
\newcommand{\F}{\mathscr{F}}
\newcommand{\cO}{\mathscr{O}}
\newcommand{\K}{\mathcal{K}}

\renewcommand{\L}{\mathcal{L}}

\newcommand{\p}{\mathfrak{p}}
\newcommand{\q}{\mathfrak{q}}
\newcommand{\m}{\mathfrak{m}}
\newcommand{\CS}{\mathcal{S}}
\newcommand{\T}{\mathsf{T}}
\newcommand{\Sp}{\mathrm{Sp}}
\renewcommand{\mod}[1]{\mathrm{Mod}_{#1}}
\newcommand{\calg}[1]{\mathrm{CAlg}_{#1}}

\newcommand{\cell}[1]{{#1}\text{-}\mathrm{cell}\text{-}}

\renewcommand{\phi}{\varphi}

\newcommand{\Q}         {{\mathbb{Q}}}
\newcommand{\Z}         {{\mathbb{Z}}}

\newcommand{\Hom}       {\operatorname{Hom}}

\newcommand{\supp}{\operatorname{supp}}

\newcommand{\cptors}{\C_{\mathrm{pre}\text{-}\mathrm{t}}}
\newcommand{\cpstan}{\C_{\mathrm{pre}\text{-}\mathrm{s}}}

\newcommand{\cptorse}{\cptors^e}
\newcommand{\cpstane}{\cpstan^e}

\newcommand{\bZp}{\mathbb{Z}_p^{\wedge}}
\newcommand{\cE}{\mathcal{E}}

\newcommand{\cEi}{\cE^{-1}}
\newcommand{\bbT}{\mathbb{T}}

\newcommand{\tensor}{\otimes}

\newcommand{\stan}{\mathrm{stan}}
\newcommand{\spcc}{\mathrm{Spc}^{\omega}} 
\newcommand{\tors}{\mathrm{tors}}
\newcommand{\downcl}{\land}
\newcommand{\upcl}{\lor\!}
\newcommand{\gen}{\mathfrak{g}}
\newcommand{\sm}{\wedge}

\newcommand{\thick}{\mathrm{Thick}}
\newcommand{\tthick}{\mathrm{Thick}_{\tensor}}
\newcommand{\loc}{\mathrm{Loc}}
\newcommand{\tloc}{\mathrm{Loc}_{\tensor}}
\newcommand{\map}{\mathrm{map}}

\newcommand{\ctors}{\C_{t}}

\title{Torsion models for tensor-triangulated categories:\\the one-step case}
\author{Scott Balchin}
\address[Balchin]{Max Planck Institute for Mathematics, Vivatsgasse 7, 53111 Bonn, Germany}
\email{balchin@mpim-bonn.mpg.de}
\author{J.P.C. Greenlees}
\address[Greenlees]{Warwick Mathematics Institute, Zeeman Building, Coventry, CV4 7AL, UK}
\email{john.greenlees@warwick.ac.uk}
\author{Luca Pol}
\address[Pol]{Fakult\"{a}t f\"{u}r Mathematik, Universit\"{a}t Regensburg, Universit\"{a}tsstra{\ss}e 31, 93053 Regensburg, Germany}
\email{luca.pol@ur.de}
\author{Jordan Williamson}
\address[Williamson]{Department of Algebra, Faculty of Mathematics and Physics, Charles University in Prague, Sokolovsk\'{a} 83, 186 75 Praha, Czech Republic}
\email{williamson@karlin.mff.cuni.cz}

\begin{document}
\begin{abstract}
Given a suitable stable monoidal model category $\C$ and a
specialization closed subset $V$ of its Balmer spectrum one can produce a
Tate square for decomposing objects into the part supported over $V$ and
the part supported over $V^c$ spliced with the Tate object. Using
this one can show that $\C$  is Quillen equivalent
to a model built from the data of local torsion objects, and the
 splicing data lies in a rather rich category. As an application, we promote the 
torsion model for the homotopy category of rational circle-equivariant
spectra from \cite{Greenlees99} to a Quillen equivalence. In addition, a close analysis of the one-step case highlights important features needed for general torsion models which we will return to in future work.
\end{abstract}

\maketitle
\setcounter{tocdepth}{1}
\tableofcontents

\section{Introduction}
The goal of this series of papers is to prove the existence of a model for a
tensor-triangulated category $\T$ which is built from local torsion
objects. When $\T$ is the derived category $\Dsf(R)$ of a commutative
Noetherian ring $R$, this corresponds to building a module from its
local cohomology at each prime, and  when working in chromatic homotopy
theory this corresponds to building a spectrum from its monochromatic
pieces. The main observation is that the splicing data lies over
completed rings, making it more algebraic.   
The motivating example is the torsion model of \cite{Greenlees99}
which is an algebraic counterpart of reconstructing a rational
circle-equivariant spectrum from pieces with a single geometric
isotropy group, and this work began with a view to extending this model to higher
dimensional tori. 

\subsection{Overview}
Given a tensor-triangulated category $(\T, \otimes, \1)$ one can consider the Balmer
spectrum $\spcc(\T)$ of its compact objects, which is a spectral topological space 
categorifying the Zariski spectrum of a commutative ring. This allows
one to apply methods from
commutative algebra in more general contexts. The reader may wish to
consider the special case $\T=\Dsf(R)$ for a commutative Noetherian
ring $R$ whilst reading the motivation. We will give the general
details later. 

In reconstructing an $R$-module $M$ there are several approaches. The
\emph{Zariski standard model} reconstructs $M$ from its localizations $L_{\wp}M:=M_{\wp}$ at primes
$\wp$, giving the equivalence between $R$-modules and quasi-coherent sheaves 
over $\spcc(\T)$ where the rather complicated reconstruction data is encoded
in the topology on the \'etale space of the sheaf. This has the
feature that $L_{\wp}M$ will usually have support spread over several
primes. 

We may instead assemble
a module from its torsion pieces supported at individual primes; in
effect, we reconstruct $M$ from its local cohomology $L_{\wp}\Gamma_{\wp}M$
(essentially $H^*_{\wp}(M_{\wp})$) at the various primes. In other words, 
we reconstruct a module $M$ from a version of its Cousin complex. We call this the \emph{Zariski torsion model}. 

As an alternative to using localizations,  we may reconstruct $M$ from
its localized completions $(M_{\wp}^{\wedge})_{\wp}$ as in
\cite{BalchinGreenlees}, which we call the \emph{adelic standard
  model} (properly speaking, if $M$ is not compact we use $L_\wp
\Lambda_\wp \1 \tensor M$; the model based on $L_\wp\Lambda_\wp M$ is
the \emph{adelic categorical model} which will not concern us here). 
The advantage is two fold. First, we may use  adelic
structure to stick the pieces together, which is simpler and more
algebraic than for the Zariski standard model, and second the localized completed rings are more likely to
be algebraic, or even formal. 

Finally, there is the \emph{adelic torsion model}, which is the subject of this paper: we are once again
reconstructing $M$ from its local cohomology, but now working over
completed rings. Since local cohomology does not notice completions,
the pieces themselves are essentially the same as for the torsion
version of the Zariski model, but now the reconstruction takes place
in a more algebraic environment: it is this richer version that we
call the adelic torsion model. Henceforth we will refer to the adelic
standard and torsion models simply as the standard and torsion models, 
unless we wish to emphasize the distinction between the adelic and the Zariski approach.

The standard and torsion approaches have different advantages and
disadvantages. The standard approach behaves well for monoidal
structures. Since completion is exact on finitely generated
modules,  the splicing data in the standard approach is often represented at the abelian
category level and one hopes to have lower homological dimension in
algebraic versions. 
On the other hand objects tend to be quite large from a naive point of
view. Objects in the torsion model tend to be smaller, but monoidal
information is lost. Furthermore,  the local cohomology of nice modules
is in the top cohomological degree, so the splicing information in the torsion model
tends to be only visible at the derived level and the homological
dimension in algebraic versions is larger. 

At this point, it is worth considering the simple example of the derived category 
$\Dsf(\Z_{(p)})$ of the $p$-local integers. We consider the
reconstruction of  a $\Z_{(p)}$-module $M$ in each of the four models.
\begin{itemize}
\item In the Zariski standard model, since $\Z_{(p)}$  is a local ring, the module itself occurs as
the stalk over $p$, so the Zariski model does not simplify anything. 
\item In the Zariski torsion model we consider $M$ as 
built from $M\tensor \Q$ (the contribution at $(0)$) and $M\tensor 
\Q /\Z_{(p)}$ (derived tensor, the contribution at $(p)$). 
\item In the adelic standard model we 
consider $M$ built from $M\tensor \Q$ (the contribution at $(0)$) and $M\tensor 
\Z_p^{\wedge}$ (the contribution at $(p)$). 
\item In the adelic categorical model we 
consider $M$ built from $M\tensor \Q$ (the contribution at $(0)$) and
$\Lambda_p M$ (the contribution at $(p)$). 
\item In the adelic torsion model we again consider $M$ as 
built from $M\tensor \Q$ (the contribution at $(0)$) and $\Lambda_pM\tensor 
\Q /\Z_{(p)}=M\tensor 
\Q /\Z_{(p)}$,  but now the latter is viewed as a module over the
$p$-adic integers. 
\end{itemize}

Looking at an example from equivariant homotopy illustrates the
distinction between the Zariski torsion model and the adelic torsion
model more starkly, but we will return to that when we have more
language.

\subsection{The dimension filtration}
It is natural  to mimic algebraic geometry, filtering the space $\spcc(\T)$ by the dimension of its primes
$$\emptyset=\spcc(\T)_{\leq -1}\subseteq \spcc(\T)_{\leq 0}\subseteq \spcc(\T)_{\leq 1}\subseteq \cdots \subseteq
\spcc(\T)$$
where $\spcc(\T)_{\leq s}$ consists of those primes of dimension $\leq
s$.  Following \cite{BalmerFiltrations} and~\cite{Stevenson} the category $\T$ has a
corresponding filtration $\T_0\subseteq  \T_1 \subseteq \T_2\subseteq
\cdots \subseteq \T$  by dimension of support. 
In any case, the set
$\spcc(\T)_{\leq s}\setminus \spcc(\T)_{\leq s-1}$ consists of primes
of exact dimension $s$. With the torsion approach,  the Verdier quotient
$\T_{s}/\T_{s-1}$ splits as a  sum of
contributions from primes of exact dimension $s$. For the torsion
model, the part over a prime $\p$ is $\p$-local and torsion.  

In view of this splitting, our main focus is on how to splice together the part of $\T$
coming from   primes of dimension $s$ to  the part of $\T$ supported
 on the generalization closed subset $U=\spcc(\T)_{>s}$ of primes of
 dimension  $>s$. It is in this splicing that the distinction between
 the Zariski and adelic torsion models appears. We will focus our
 attention on categories where the Balmer spectrum is  Noetherian since the splicing is
 significantly simpler in that case.

In the present paper we consider  the case when there 
is just one step in the assembly process. 
We will follow this process through for the torsion model
in  a higher dimensional Noetherian setting in \cite{torsion2}. 
Ensuring that the splicing data is algebraic requires us to refine 
the  naive iterative approach. 

When the Balmer spectrum is not Noetherian, the Zariski topology is no
longer generated by the closures of points, bringing a significant additional 
layer of complication: the general framework  
for assembling models from a good filtration of the Balmer 
spectrum in this context is studied in \cite{prismatic}. 

\subsection{The one-step context}
When we are dealing with just one step, it is helpful to consider a
more general context so we can identify particular features which lead
to good behaviour. 

We begin with a smashing localization $L$ giving the context of
\cite{Greenlees01b} from which we can construct our basic framework. 
The associated torsion functor is  the fibre $\Gamma $, and this then also
gives the associated completion $\Lambda =\Hom (\Gamma \1, -)$. Altogether we have the diagram 
\[\begin{tikzcd}
\Gamma \1 \arrow[d, "\simeq"']\arrow[r]& \1 \arrow[d]\arrow[r]& L\1 \arrow[d] \\
\Gamma \Lambda \1 \arrow[r]& \Lambda \1 \arrow[r] & L\Lambda\1 
\end{tikzcd} 
\]
in which the right hand square is a homotopy pullback.

For the motivating example in our general programme, we choose a subset $U$ of the spectrum $\spcc(\T)$ closed under generalization, such as the set of primes of 
dimension $> s$. One may then form a smashing localization $X \to 
LX$ supported at $U$.

It may be helpful to consider three running examples to illustrate the discussion.
\begin{enumerate} 
\item The derived category 
$\Dsf(\Z_{(p)})$ of the $p$-local integers, where we take $L$ to be rationalization. Equivalently this  is given by choosing $U$ to consist of the zero ideal.
\item The $E(n)$-local 
category of spectra $L_{n}\Sp$ (implicitly $p$-local). Here we take $L$ to be the smashing localization at the Johnson--Wilson theory $E(n-1)$, which corresponds to setting $U$ to be the collection of primes of dimension $>0$.
\item The category $\Sp_G$ of $G$-equivariant spectra for $G$ a compact Lie 
group. Here we choose a family $\F$ of subgroups and take $L$ to be the smashing
localization making the underlying 
spectrum $\F$-contractible.   
\end{enumerate}
\begin{center}
\begin{tabular}{c|c|c|c|c|c}
$\T$ & $\1$ & $L\1$ & $\Gamma\1$ & $\Lambda\1$ & $L\Lambda\1$ \\
\hline 
$\Dsf(\Z_{(p)})$ & $\Z_{(p)}$ & $\mathbb{Q}$ & $\Sigma^{-1}\mathbb{Q} /\Z_{(p)}$ & $\mathbb{Z}_p^\wedge$ & $\Q \otimes \Z_p^\wedge$ \\
& & & & \\
$L_{n}\Sp$ & $L_{n}S^0$ & $L_{n-1}S^0$ & $M_nS^0$ & $L_{K(n)}S^0$ & $L_{n-1}L_{K(n)}S^0$ \\
& & & & \\
$\Sp_G$ & $S^0$ & $\widetilde{E}\F$ & $E\F_+$ & $DE\F_+$ & $\widetilde{E}\F \wedge DE\F_+$ \\
\end{tabular}
\end{center}
For more details on these examples, see Section~\ref{sec-example}. 

The above diagram shows that the \emph{Tate square} 
\[\begin{tikzcd}
\1 \arrow[r] \arrow[d] & L\1 \arrow[d] \\
\Lambda \1 \arrow[r] & L\Lambda\1
\end{tikzcd} \] 
expresses the unit $\1$ of $\T$  as the homotopy pullback 
of localized completed pieces. Moreover, the objects occurring in the
Tate square are still monoid objects in $\T$, and the square
yields the standard model of $\T$ in terms of modules over the
localized-completed pieces~\cite{BalchinGreenlees}. The price to pay for
having monoids is that the 
completion functor involves some `blurring' around $V$ so
that  the objects may have contributions from  open sets containing $V$
(for example $\bZp$ has non-trivial rationalization).  
Indeed, if the completion had support precisely at $V$, then the Tate object $L\Lambda\1$ would be equivalent
to the zero object and $\1$ would be a product. 
The fact that the Tate object $L\Lambda\1$ is
non-zero is because of the blurring appearing in the completion
$\Lambda$. 

The current paper gives a related model built instead from pieces
which are torsion in the sense of the functor $\Gamma$. When we use the dimension
filtration in \cite{torsion2} the pieces will each be supported at a
single prime. 

The diagram 
\[\begin{tikzcd}
& L \1 \arrow[d] &\\
\Lambda\1 \arrow[r] & L\Lambda\1 \arrow[r]&\Sigma\Gamma \Lambda
\1\simeq \Sigma \Gamma \1
\end{tikzcd} 
\]
shows the relationship between the standard model and the torsion
model.  The standard model comes from
the left hand three objects. 
  The torsion model comes from the
right hand three objects. This is called the torsion model because the
whole category relative to the local category is controlled by the
torsion object $\Gamma \1$. 

The two models determine each other. The three objects in the standard model  
give the Tate square by taking pullbacks, and hence the whole
diagram by taking fibres. The three objects in the torsion model give
the left hand object by taking the fibre of the horizontal and then one reconstructs the whole diagram as before.

\subsection{The torsion model}
Suppose that $\T$ is a tensor-triangulated category which is the
homotopy category of a well behaved stable monoidal model
category $\C$, see Section~\ref{sec:models} for detailed conditions. 

We now give an informal statement of the one-step torsion model. The model
is in terms of the sections of the diagram 
$$\begin{tikzcd}
 \mod{L\1} \arrow[d, "\bullet" description] &\\
 \mod{L\Lambda\1} \arrow[r] &
                                \mod{\Lambda\1} .
  \end{tikzcd}$$
This consists of an object in each vertex category together with maps
relating these (after suitable extension of
scalars). The dot on the vertical map indicates that the object
in the bottom left category is obtained from that in the top left by
extension of scalars, see Section~\ref{sec:diagrams}. More precisely, an object of this diagram is  a $\Lambda \1$-module map  $\beta \colon L \Lambda \1 \otimes_{L \1} V \to T$ where $V$ is a $L \1$-module and $T$ is a $\Lambda \1$-module.

 The main theorem states that the model category $\C$ 
is Quillen equivalent to a cellularization of the category  of
sections which introduces the torsion character of the model. The following theorem specializes Theorem~\ref{thm:torsionmodel} and Theorem~\ref{thm:bifibrant} in the main body of the paper.

\begin{thm*}
There is a Quillen equivalence $$\C \simeq_Q \mathrm{cell}\text{-}
  \left(
  \begin{tikzcd}
   \mod{L\1} \arrow[d, "\bullet" description] & \\
\mod{L\Lambda\1} \arrow[r] 
                               & \mod{\Lambda\1} 
  \end{tikzcd}
  \right).$$
The effect of the cellularization is that an object of the
tensor-triangulated category $h\C$ is equivalent to one specified by:
\begin{itemize}
\item an $L\1$-module $V$;
\item a $\Lambda\1$-module $T$ which is torsion in the sense that $\Gamma T \simeq T$;
\item a $\Lambda\1$-module map $L\Lambda\1 \otimes_{L\1} V \to T$.
\end{itemize}
\end{thm*}


 It is rare  for the monoids appearing in the model to have formal homotopy. However, one can hope to use the torsion model as a first step towards an algebraic classification, such as the construction of an Adams spectral sequence. 
 
In very favourable cases the torsion model may be equivalent to  the derived category of
an abelian category. Currently the best examples of this come from rational
equivariant cohomology theories. With this example in mind we can highlight the advantage of the adelic torsion model over the Zariski torsion model.

The Zariski torsion model corresponds to the usual isotropy separation filtration
(constructing $X$ from $X\sm E\F_+$ and $X\sm \widetilde{E}\F$). This is
a very powerful method, but it is not very algebraic  because the sphere is
very far from being formal, and its coefficient ring is very
complicated. By contrast, the adelic torsion model does the same thing
relative to the Tate construction on the sphere
(constructing $X$ from $X\sm E\F_+$ and $X \sm DE\F_+\sm
\widetilde{E}\F$). These spectra are both modules over the ring
$DE\F_+$. If we are working rationally, then $DE\F_+$ has well
understood homotopy and is formal, so working over this ring is a big
advantage.

\subsection{Rational $\bbT$-spectra}

For any compact Lie group $G$, the
second author conjectured the existence of a graded abelian category
$\A(G)$ together with a Quillen equivalence $$\Sp_G/\mathbb{Q}
\simeq_Q d\A(G)$$ between rational $G$-spectra and differential
objects in $\A(G)$. We call  the category $\A(G)$ an
\emph{abelian model}, and $d\A(G)$ the \emph{algebraic model}.

For the circle group $G = \mathbb{T}$,  two abelian categories
$\A(\mathbb{T})$ and $\A_t(\mathbb{T})$ are defined in
\cite{Greenlees99} as follows.  Let $\cO_{\F}=
\prod_{F \in \F} H^*(B\mathbb{T}/F)$ where $\F$ is the family of
finite subgroups of $\mathbb{T}$.
Inverting the Euler classes $\cE$ gives a ring $t_\F = \cE^{-1}\cO_\F$. We define
$$\A(\mathbb{T}) = \{\beta\colon N \to t_\F \otimes V \mid
\text{$\beta$ is inversion of $\cE$}\} \qquad \A_t(\mathbb{T}) = \{t_\F \otimes V \to T\}$$
where $V$ is a graded $\mathbb{Q}$-vector space, $N$ is an
$\cO_\F$-module and $T$ is an Euler torsion $\cO_\F$-module. For more
details see Section~\ref{sec:Tspectra}. It is  also shown in
\cite{Greenlees99} 
that the derived categories of $\A(\mathbb{T})$ and $\A_t(\mathbb{T})$
are both equivalent to the homotopy category of rational
$\mathbb{T}$-spectra. The category $\A(\mathbb{T})$ is called the
\emph{standard model} for rational $\mathbb{T}$-spectra, and the
category $\A_t(\mathbb{T})$ is called the \emph{torsion model} for
rational $\mathbb{T}$-spectra. The standard model has better monoidal
properties and lower homological dimension, but the torsion model is
often more approachable  (see~\cite[Part II]{Greenlees99}).

The equivalence $h\Sp_\mathbb{T} \simeq \Dsf(\A(\mathbb{T}))$ between
rational $\mathbb{T}$-spectra and the derived category of the standard
model has since been promoted to a Quillen
equivalence~\cite{GreenleesShipley18}, shown to be monoidal in 
\cite{BarnesGreenleesKedziorekShipley17}. The general
theory developed in this paper shows the equivalence $h\Sp_\mathbb{T}
\simeq \Dsf(\A_t(\mathbb{T}))$ between rational $\mathbb{T}$-spectra
and the torsion model can be promoted to a Quillen equivalence, which moreover proves that the equivalence is triangulated. The following theorem appears in the main body of the paper as Theorem~\ref{thm-torsion model for T-spectra}.

\begin{thm*}
There is a  Quillen equivalence $$\Sp_\mathbb{T} \simeq_Q d\A_t(\mathbb{T})$$ and hence a triangulated equivalence $h\Sp_{\mathbb{T}} \simeq_{\triangle} \Dsf(\A_t(\mathbb{T}))$.
\end{thm*}

The proof of the equivalence $h\Sp_\mathbb{T} \simeq
\Dsf(\A_t(\mathbb{T}))$ given in~\cite{Greenlees99} 
passes through the equivalence $h\Sp_\mathbb{T} \simeq
\Dsf(\A(\mathbb{T}))$ to the standard model, obscuring the relation
between the topology and the abelian torsion model
$\A_t(\mathbb{T})$. Our proof is more direct: it proceeds by splitting
$\mathbb{T}$-spectra into its $\F$-local and $\F$-torsion parts and
showing the resulting model is formal. 

In future work \cite{torsion2}, we will describe a torsion model for rational
$G$-spectra when $G$ is a torus of arbitrary rank. 

\subsection*{Conventions}
We will follow the convention of writing the left adjoint on the top in
adjoint pairs displayed horizontally, and on the left in adjoint pairs
displayed vertically.  Weak equivalences have a direction, but weak
equivalence of objects is the resulting equivalence
relation. Similarly for Quillen equivalences. Given a (graded or
ungraded) ring $R$ we write $\mod{R}$ for the category of
dg-$R$-modules. All cellularizations are at stable sets of objects.

\subsection*{Acknowledgements}
The first and second authors were supported by the grant EP/P031080/2. The third author thanks the SFB 1085 Higher Invariants in Regensburg for support. The fourth author was supported by the grant GA~\v{C}R 20-02760Y from the Czech Science Foundation.

The first and second authors are grateful to Tobias Barthel for
numerous conversations about various models based on the Balmer
spectrum.  The authors would also like to thank the Isaac Newton Institute for Mathematical Sciences, Cambridge, for support and hospitality during the programme \emph{$K$-theory, algebraic cycles and motivic homotopy theory}.

\section{Diagrams of module categories}

  Throughout the paper we  freely use the language of model categories
  and use~\cite{Hovey99} as a reference.
  We briefly recall some results on categories of modules and on diagrams of model 
  categories as they will be used repeatedly throughout.  
  
\subsection{Module categories}

  Let $\C$ be a monoidal model category and let $R$ be a monoid in $\C$. 
  Denote by $\mod{R}(\C)$ the category of (left) $R$-modules in $\C$. 
  If the ambient category is clear, we will omit it from the notation. 
  We will be interested in two model structures on $\mod{R}(\C)$:
  \begin{itemize}
  \item \emph{projective model structure:} the weak equivalences and fibrations are 
  determined in the underlying category $\C$. This exists under mild conditions on $\C$, 
  see~\cite{SchwedeShipley00}.
  \item \emph{injective model structure:} the weak equivalences and cofibrations are 
  determined in the underlying category $\C$.  This exists under more stringent hypotheses on 
  $\C$, see~\cite{HessKedziorekRiehlShipley17}.
  \end{itemize}
  
  We will be particularly interested in the following algebraic cases. 

\begin{ex}
  Let $R$ be a graded ring. The category of dg-modules over $R$ supports a projective model 
  structure where the weak equivalences are the quasi-isomorphisms, the fibrations are the 
  surjections, and the cofibrations are the monomorphisms which are
  split on the underlying graded modules and which 
  have dg-projective cokernel. See~\cite{Abbasirad} for details. 
\end{ex}

\begin{ex}
  The category of dg-modules over $R$ also supports an injective model structure where the 
  weak equivalences are the quasi-isomorphisms and the cofibrations are the monomorphisms, 
  see~\cite[2.3.13]{Hovey99} for example. This is important for our paper for the following reason. 
  For a finitely generated homogeneous ideal $I$, let us consider the full subcategory 
  of $I$-power torsion modules.   One observes that the category of torsion 
  modules is abelian but does not have enough projectives, in particular, it does not support a projective model structure.   However, one can show that it has enough injectives and so it does support a injective
  model structure, see~\cite[2.3.13]{Hovey99} and \cite[8.6]{GreenleesShipley11}.
\end{ex}

Given a map of monoids $\theta\colon S \to R$ in $\C$ there is an adjoint triple
\[
\begin{tikzcd}
\mod{S} \arrow[rr, yshift=3.5mm, "\theta_\ast" description] \arrow[rr, yshift=-3.5mm, "\theta_!" description] & & \mod{R} \arrow[ll, "\theta^\ast" description]
\end{tikzcd}
\]
given by restriction of scalars $\theta^*$, extension of scalars $\theta_* = R \otimes_S -$ and coextension of scalars $\theta_! = \Hom_S(R,-)$. This triple of functors is not always Quillen as the following result shows. 

\begin{lem}[{\cite[3.8]{changeofgroups}}]\label{lem:restrictionleftQuillen}
Let $\theta\colon S \to R$ be a map of monoids in a monoidal model category $\C$. The extension of scalars functor $\theta_\ast$ is always left Quillen in the projective model structure. If the unit of $\C$ is cofibrant, then the restriction of scalars functor $\theta^\ast$ is left Quillen in the projective model structure if and only if $R$ is cofibrant as an $S$-module.
\end{lem}

\subsection{Diagrams of model categories}\label{sec:diagrams}
We recall the necessary background from~\cite{GreenleesShipley14b}. 

\begin{defn}
Let $\Dsf$ be a small category. A \emph{diagram of model categories} of shape $\Dsf$ is a collection of model categories $\C(d)$ for each $d \in \Dsf$ together with a left Quillen functor $\C(f)\colon \C(d) \to \C(d')$ for each morphism $f\colon  d \to d'$ in $\Dsf$, 
which are strictly compatible with composition. Equivalently, this is a functor from $\Dsf$ into the category of model categories and left Quillen functors. 
\end{defn}

\begin{defn}\label{defn:leftsections}
Let $\C$ be a diagram of model categories of shape $\Dsf$. The \emph{category of  sections} of $\C$ consists of an object $X_d \in \C(d)$ for each $d \in \Dsf$ together with maps $\C(f)(X_d) \to X_{d'}$ in $\C(d')$ for each morphism $f\colon d \to d'$ in $\Dsf$, which are strictly compatible with composition.
\end{defn}

We make a particular case of interest explicit.

\begin{ex}\label{ex-left section modules}
  Let $T \xrightarrow{\theta} S \xleftarrow{\phi} R$ be a diagram of monoids in a monoidal 
  model category. If we assume that $S$ is cofibrant as an $R$-module,
  then using the projective model structure in all three cases,   
  Lemma~\ref{lem:restrictionleftQuillen} ensures that
  \[
  \mod{T} \xrightarrow{\theta_\ast} \mod{S} \xrightarrow{\phi^\ast} \mod{R}
  \] 
  is a diagram of model categories.
  The category of  sections consists of the following data:
\begin{itemize}
\item a $T$-module $L$;
\item an $S$-module $M$;
\item an $R$-module $N$;
\item an $S$-module map $\theta_\ast L \to M$;
\item an $R$-module map $\phi^\ast M \to N$.
\end{itemize} 

\end{ex}

The category of  sections of $\C$ supports a `diagram injective' model structure under some combinatorial conditions on the indexing category. Recall that $\Dsf$ is called an \emph{inverse category} if, informally, the morphisms in $\Dsf^{\mathrm{op}}$ only `go in one direction'.  
We refer the reader to~\cite[5.1.1]{Hovey99} for the definition, and here instead give two main examples:
\[
\bullet \to \bullet \leftarrow \bullet \quad \mathrm{and} \quad \bullet \to \bullet \to \bullet.
\]

  We have the following result.

\begin{thm}[{\cite[3.1]{GreenleesShipley14b}}]
  Let $\C$ be a diagram of model categories of shape $\Dsf$, where $\Dsf$ is an inverse category. 
  There is a diagram injective model structure on the category of  sections of $\C$ 
  with weak equivalences and cofibrations determined objectwise. Moreover if each $\C(d)$ is 
  cellular and proper then so is the diagram injective model structure on $\C$.
\end{thm}

\begin{nota}\label{nota:dimp}
  We will often consider diagrams of module categories as in 
  Example~\ref{ex-left section modules}. The category of  sections admits a \texttt{dimp} 
  (diagram injective, module projective) model structure and a \texttt{dimi} 
  (diagram injective, module injective) model structure.
\end{nota} 

\section{Tensor-triangulated categories and their enhancements}
\subsection{Balmer spectrum}

  We recall the necessary background from~\cite{Balmer05}.
  
\begin{defn}
  A \emph{tensor triangulated category} consists of a triangulated category $\T$ and a 
  symmetric monoidal tensor product $\otimes \colon \T \times \T \to \T$ which is exact in 
  each variable. We additionally assume that $\T$ is closed symmetric monoidal and 
  denote the internal hom by $\underline{\mathrm{Hom}}$.
\end{defn}

  A full subcategory of $\T$ is \emph{thick} if it is closed under suspensions, completing 
  triangles and taking retracts. It is \emph{localizing} if it is thick and closed under arbitrary sums.
  Finally, it is is an \emph{ideal} if it is closed under tensoring with 
  arbitrary objects of $\T$. For a set of objects $\K$ in $\T$, we
  write $\mathrm{Thick}(\K)$ for the smallest thick subcategory of
  $\T$ containing $\K$, and $\mathrm{Loc}(\K)$ for the smallest
  localizing subcategory of $\T$ containing $\K$. We write
  $\mathrm{Thick}_\otimes(\K)$  for the smallest thick subcategory of
  $\T$ containing $\K$ and closed  under tensoring with compact objects. We write
$\mathrm{Loc}_\otimes(\K)$ for the smallest localizing subcategory of
$\T$ containing $\K$ and closed under tensoring with arbitrary
objects. If $\1$ is a compact generator for $\T$ then $\thick(\K)=\tthick(\K)$ and
$\loc (\K)=\tloc(\K)$.
 
  To avoid set theoretical issues we will restrict to 
  essentially small tensor triangulated categories. Concretely, this means that given a  
  tensor triangulated category $\T$, we will restrict to the full subcategory 
  $\T^{\omega}$ of compact objects; that is, those objects $X$ of $\T$ for which the functor 
  $[X,-]$ preserves arbitrary sums. Recall that an object $X$ of $\T$ is said to be dualizable if the natural map $\underline{\Hom}(X,\1) \otimes Y \to \underline{\Hom}(X,Y)$ is an isomorphism for all $Y \in \T$.  In general $\T^{\omega}$ will not always be essentially 
  small; however, it is when $\T$ is \emph{rigidly-compactly generated}. This means that $\T$ has 
  a set of compact and dualizable objects $\G$ that generates the whole category in the sense that $\mathrm{Loc}(\G) = \T$, and the monoidal unit $\1$ is compact. It follows that the compact objects and the dualizable objects in $\T$ coincide, see~\cite[2.1.3(d)]{HoveyPalmieriStrickland97}.
  
  We list a few key examples of rigidly-compactly generated tensor 
  triangulated categories. 
  
\begin{exs}\leavevmode
  \begin{itemize}
   \item[(a)] The derived category $\Dsf(R)$ of a commutative ring is a tensor triangulated 
   category with respect to the (derived) tensor product of $R$-modules. 
   It is compactly generated by $R$, and the 
   full subcategory of compact objects coincides with that of perfect complexes, i.e., 
   those complexes that are quasi-isomorphic to a bounded complex of finitely generated 
   projectives.
   \item[(b)] The stable homotopy category $h\Sp$ is a tensor triangulated category with 
   respect to the (derived) smash product. 
   It is compactly generated by the sphere spectrum $S^0$, and 
   the full subcategory of compact objects coincides with that of finite spectra, that is, 
   spectra of the form $\Sigma^{n}A$ where $A$ is a homotopy retract of a finite based $CW$-complex and $n$ is an 
   integer.
   \item[(c)] For $G$ a compact Lie group, the equivariant stable homotopy category $h\Sp_G$ is a tensor triangulated 
   category with respect to the (derived) smash product. It is compactly generated by 
   the set of cosets $G/H_+$ as $H$ ranges over the closed subgroups
   of $G$. The compact objects are those of the form $S^{-V} \wedge
   \Sigma^{\infty}A$ for $V$ a $G$-representation and $A$ a homotopy
   retract of a based finite $G$-CW
   complex~\cite[I.4.7]{LewisMaySteinbergerMcClure}. One category of
   particular interest in this paper comes from considering the
   rational version of this category, that is, the category obtained
   by localizing with respect to the rational equivalences.
   \item[(d)] Let $G$ be a finite group and $k$ be a field with characteristic dividing the order of $G$. The stable module category $\mathrm{StMod}(kG)$ is a tensor 
   triangulated category with respect to the tensor product $- \otimes_k -$ with diagonal $G$-action. It is compactly 
   generated by the simple modules, and the full subcategory of compact objects coincides 
   with the class of finitely generated $kG$-modules.
  \end{itemize}
\end{exs}
 
\begin{defn}\label{defn:support}
  Let $\T$ be a tensor triangulated category.
  \begin{itemize}
  \item
	A \emph{prime ideal} is a proper thick ideal $\p$ with the property that 
	\[
	x \otimes y \in \p \Rightarrow x \in \p \; \text{or}\; y \in \p.
	\] 
  \item The \emph{Balmer spectrum} $\spcc (\T)$ is the set of all 
	prime ideals of $\T^{\omega}$. 
  \item We endow the Balmer spectrum with a \emph{Zariski topology} generated by the closed 
    sets 
	\[
	\supp(X) = \{\p \in \spcc (\T) \mid X \not\in \p\}
	\]
    where $X$ ranges over the compact objects of $\T$. In particular, the 
    \emph{specialization closure} of a prime $\p$ is given by 
	$\overline{\{\p\}} = \{\q \mid \q \subseteq \p\}.$
\end{itemize}
\end{defn}

We recall the following important result that we will use repeatedly throughout.

\begin{prop}[{\cite[2.6]{Balmer05}}]\label{prop:properties_support}
 Let $\T$ be a rigidly-compactly generated tensor triangulated category.
  The support of compact objects of $\T$ has the following properties:
	\begin{itemize}
	\item[(a)] $\supp(0)= \emptyset$ and $\supp(\mathbbm{1})=
          \spcc (\T)$;
	\item[(b)] $\supp(X \oplus Y)= \supp(X) \cup \supp(Y)$;
	\item[(c)] $\supp(\Sigma X)= \supp(X)$;
	\item[(d)] $\supp(Y) \subseteq \supp(X) \cup \supp(Z)$ for any distinguished triangle 
	$X\to Y\to Z\to\Sigma X$;
	\item[(e)] $\supp(X \otimes Y)= \supp(X) \cap \supp(Y)$. 
	\end{itemize}
  Moreover, two compact objects have equal support if and only if they generate the same 
  thick ideal. In particular, if an object has empty support it is equivalent to the zero object.
\end{prop}

  We will only consider Balmer spectra which are \emph{Noetherian} spaces, that is those 
  which satisfy the ascending chain condition for open subsets.
  By~\cite[7.14]{BalmerFavi11}, the Balmer spectrum is Noetherian if and only if for each 
  prime $\p$ there exists a compact object $K_\p$ in $\T$ such that 
  $ \supp(K_\p) = \overline{\{\p\}}$.  
  We will call $K_\p$ a \emph{Koszul object for $\p$}.
  
  The \emph{dimension} of a prime is the length of a maximal length chain of 
  primes down to a closed point (= minimal prime). 
  The dimension of the Balmer spectrum is the supremum of the
  dimensions of all primes. We note that such a chain always exists as the space is
  spectral~\cite{BuanKrauseSolberg07}.

\subsection{Sufficiently nice model categories}\label{sec:models}
Our methods require that our tensor-triangulated category is the
homotopy category of a suitable stable monoidal model category.
Many of the finer details will be omitted since they are not relevant
for the remainder of the paper, but we recall essentials from~\cite[\S 4]{BalchinGreenlees}.

Let $\C$ be a model category. The first requirement is that the
homotopy category $h\C$ is tensor-triangulated. If $\C$ is a stable,
symmetric monoidal model category this is true~\cite[4.3.2,~6.6.4,~\S7.1]{Hovey99}. We will moreover assume that the model structure on $\C$ satisfies the \emph{monoid axiom} which ensures that categories of modules over monoids in $\C$ support a projective model structure, see~\cite{SchwedeShipley00}. Next, we need to be able to perform left Bousfield localizations and cellularizations (i.e., right Bousfield localizations). In order to do this, we assume that $\C$ is proper and cellular. It then follows by~\cite[4.1.1,~5.1.1]{Hirschhorn03} that left Bousfield localizations at sets of maps, and cellularizations at sets of objects exist. 

In order to ensure that the homotopy category $h\C$ is rigidly-compactly generated, we assume that $\C$ is compactly generated in the sense of~\cite[1.2.3.4]{ToenVezzosi08}, has unit $\1$ cofibrant and cell-compact, and has generating cofibrations of the form
$S^n \otimes \mathcal{G} \to \Delta^{n+1}\otimes \mathcal{G}$
where $\mathcal{G}$ is a suitable set of cell-compact cofibrant objects which will become the compact generators of $h\C$.
The fact that these assumptions guarantee that $h\C$ is rigidly-compactly generated follows from~\cite[1.2.3.7,~1.2.3.8]{ToenVezzosi08}. These conditions also ensure that homological localizations exist~\cite[\S 6.A]{BalchinGreenlees}.

\begin{defn}
We say that a stable symmetric monoidal model category $\C$ is \emph{rigidly-compactly generated} if it satisfies the above conditions. 
\end{defn}

In summary we have the following.

\begin{thm}
Let $\C$ be a rigidly-compactly generated model category, then its homotopy category $h\C$ is a rigidly-compactly generated tensor-triangulated category.
\end{thm}

As already alluded to in the previous section, we do not wish to consider all tensor-triangulated categories, but only those whose Balmer spectrum has a particularly tame form. This leads us to the final definition of this section.

\begin{defn}
A rigidly-compactly generated model category $\C$ is said to be a \emph{finite-dimensional Noetherian} model category if $\spcc(h\C)$ is a finite-dimensional Noetherian space.
\end{defn}

 We refer to Section~\ref{sec-example} for a discussion of some examples of finite-dimensional Noetherian model 
  categories.

\section{Localization, completion and support}\label{sec:functors}

\subsection{Localization, completion and support}
First we define localization, torsion and completion functors, and use them to define support. Additional details can be found in~\cite{BalchinGreenlees, Stevenson18}.

   Let $\C$ be a finite-dimensional Noetherian model category, and write $\G$ for a set of compact and dualizable generators. By the Noetherian
   condition, for each prime $\p$ we may choose a Koszul object
   $K_\p$ (i.e., 
   compact and with support $\overline{\{ \p\}}$). These are not
   canonical, but any two generate the same thick tensor ideal, so that
   constructions below will not depend on this choice. The theory of
   supports shows 
$$\p =\tthick (K_{\q}\mid \p \not \subseteq \q). $$

Given a collection of prime ideals $\CS$ we let $\K(\CS)$ denote a set of Koszul 
   objects for $\CS$, i.e., $   \K(\CS)= \{ K_\p \mid \p \in \CS \}. $
A collection of prime ideals $\CS$ is said to be \emph{a family}  if $\p
\in \CS$ and $\q \subseteq \p$ implies that $\q \in \CS$,
and \emph{cofamily}  if  $\p \in \CS$ and $\p \subseteq \q$ implies that $\q \in \CS$. 

\begin{rem}
We have adopted the terminology from equivariant topology for brevity.
The standard geometric terminology is that a family is `a specialization closed
set' and that a cofamily is `a generalization closed set'. 
It is helpful to observe that a  Zariski closed set is a family,  and a Zariski open
set is a cofamily.\end{rem}

In particular, we will consider the cones below and above a fixed prime 
  $\p$: 
  \[
  \downcl(\p)= \{ \q \mid \q \subseteq \p \} \quad \mathrm{and} \quad 
  \upcl(\p)= \{\q \mid \p \subseteq \q \}. 
  \]
 By definition $\downcl(\p)$ is a family  and
 $\upcl(\p)$ is a cofamily.  We note that $\downcl(\p)$ is the specialization closure of 
 $\{\p\}$.
 
 Given a family $V$, we may define associated torsion, localization and completion functors, 
 which are independent of the choice of Koszul objects $\K(V)$.
\begin{defn}
  Let $V$ be a family with complement $V^c$ the associated cofamily. 
  \begin{itemize}
   \item The \emph{torsion functor} $\Gamma_{V}$ is the
     cellularization at $\K(V) \otimes \G$. 
   \item The \emph{localization functor} $L_{V^c}$ is the nullification of $\K(V)$.
 \item The \emph{completion functor} $\Lambda_{V}$ is  the homological 
   localization at $\bigoplus_{\p \in V}K_\p$.
  \end{itemize}
\end{defn}

\begin{rem}\leavevmode
\begin{itemize}
\item[(a)] The functors $\Gamma_{\downcl(\p)}$, $\Lambda_{\downcl(\p)}$ and $L_{\upcl(\p)}$ play a particularly important role. We note that the definitions of these can be simplified as follows, by the theory of supports:
\begin{itemize}
\item $\Gamma_{\downcl(\p)}$ is the cellularization at $K_\p \otimes \G$;
\item $L_{\upcl(\p)}$ is the nullification of $K_\p$;
\item $\Lambda_{\upcl(\p)}$ is the homological localization at $K_\p$.
\end{itemize}
\item[(b)] The apparent asysmmetry between $\Gamma_V$ and $\Lambda_V$ is due to
the fact that Bousfield localization is defined in terms of an
enrichment in spaces (i.e., $\map(K_\p, X)$ is a space rather than an
object of $\C$), whereas that of a homological localization is defined
in terms of the tensor product, and $K_\p\tensor X$ is an object of
$\C$.
\item[(c)] We remark that \cite[6.7]{BalchinGreenlees} is therefore incorrect unless the category is compactly generated by its tensor unit, and should instead
refer to cellularization at $K_\p \otimes \G$.
\item[(d)] Since the homotopy category of a cellularization at a set of compact objects $\K$ is the localizing subcategory generated by $\K$ (see~\cite[2.5, 2.6]{GreenleesShipley13}), one sees that the homotopy category of $V$-torsion objects is the localizing tensor ideal generated by $\K(V)$. This puts us in the context of~\cite{Greenlees01b}. 
\end{itemize}
\end{rem}

Using these functors we can define a notion of support for non-compact objects following~\cite{BensonIyengarKrause08}. 
\begin{defn}
We define the \emph{support} of on object $M \in \C$
by $$\mathrm{supp}(M) = \{ \p \in \spcc (h\C) \mid \Gamma_{\downcl(\p)} L_{\upcl(\p)} M \not\simeq 0\}.$$
\end{defn}
Recall the notion of support for compact objects from Definition~\ref{defn:support}. We record the following key facts.

\begin{prop}[{\cite[7.17]{BalmerFavi11},~\cite[6.9(ii)]{Stevenson2013}}]
Let $\C$ be a finite-dimensional Noetherian model category.  The support defined above satisfies the following properties:
\begin{itemize}
\item[(a)] The two notions of support coincide for any compact objects of $\C$. 
\item[(b)] The support defined above satisfies the following properties:
	\begin{enumerate}
		\item For any set of objects $\{X_i\}_{i \in I}$ in $\C$ we have $\mathrm{supp}(\oplus_{i \in I} X_i) = \cup_{i \in I} \mathrm{supp}(X_i).$
		\item For all $X \in \C$ we have $\mathrm{supp}(\Sigma X) = \mathrm{supp}(X)$.
\item For any $X,Y \in \C$ there is an inclusion $\mathrm{supp}(X \otimes Y) \subseteq \mathrm{supp}(X) \cap \mathrm{supp}(Y)$. This is an equality if $X$ and $Y$ are compact objects.
		\item Given a distinguished triangle $X \to Y \to Z \to \Sigma X$ in $\C$ then there is an inclusion $\mathrm{supp}(Y) \subseteq \mathrm{supp}(X) \cup \mathrm{supp}(Z)$.
	\end{enumerate}
\item[(c)] An object has empty support if and only if it is zero.
\end{itemize}
\end{prop}

  It follows from the definition of the torsion and localization functors, that for any object $X$ of $\C$, there is a natural 
  cofibre sequence
   \[
  \Gamma_{V}X \to X \to L_{V^c}X
  \]
  with
  \[
  \supp(\Gamma_{V}X)= \supp(X) \cap V \quad \mathrm{and} \quad 
  \supp(L_{V^c}X)= \supp(X) \, \backslash V.
  \]  

   \begin{ex}\label{eg:family dim less than}
 Let $V$ consist of the primes of dimension less than or equal to $n$. In
 this case we write $\Gamma_V = \Gamma_{\leq n}$, $L_{V^c} = L_{\geq n+1}$
 and $\Lambda_V = \Lambda_{\leq n}.$ The notation is chosen as such
 for the effect on the supports; for example $\supp(\Gamma_{\leq
   n}\1)$ is the collection of primes of dimension less than or equal
 to $n$ and $\supp(L_{\geq n+1}\1)$ is the collection of primes of
 dimension greater than or equal to $n+1$. 
 \end{ex}

\subsection{Properties of the functors}

The following proposition records key facts and properties which we will use throughout this paper.
\begin{prop}\leavevmode\label{prop-tate cohomology functors} Let $V$ be a family.
\begin{itemize}
		\item[(a)] The functors $L_{V^c}$ and $\Lambda_V$ preserve monoid objects as endofunctors on $\C$. 
		\item[(b)]	The functors $\Gamma_V$ and $L_{V^c}$
                  are smashing (i.e., $\Gamma_V X \simeq \Gamma_V \1
                  \otimes X$ and $L_{V^c}X \simeq L_{V^c}\1 \otimes X$), but $\Lambda_V$ is usually not smashing.
		\item[(c)] If $R$ denotes fibrant replacement in $\C$, then 
		$\underline{\Hom}(\Gamma_V \1,RX)\simeq \Lambda_V X$.
		\item[(d)] The natural maps $\Gamma_V X \to \Gamma_V\Lambda_V X$ and $\Lambda_V\Gamma_V X \to \Lambda_V X$ are equivalences for all $X$. We call this the \emph{MGM equivalence}.
	\end{itemize}
\end{prop}

\begin{proof}
For (a), by~\cite[3.2]{White18} it is sufficient to prove that $L_{V^c}\C$ and $\Lambda_V\C$ satisfy the monoid axiom, since both $L_{V^c}$ and $\Lambda_V$ are monoidal Bousfield localizations~\cite[5.6,~6.3]{BalchinGreenlees}. Recall that $\C$ satisfies the monoid axiom by our standing assumptions. Therefore, $\Lambda_V\C$ also satisfies the monoid axiom since it is a homological localization, see for instance~\cite[3.8]{Barnes09b}. To show that $L_{V^c}\C$ satisfies the monoid axiom, we note that we can view it as a homological localization with regard to the object $L_{V^c}\1$, and so part (a) follows.
 		Part (b) is well known; see for instance~\cite[3.3.3,~3.3.5]{HoveyPalmieriStrickland97} and~\cite[5.5]{BalchinGreenlees}. Part (c) follows from the fact that every $\K(V)$-equivalence is a $\Gamma_V\1$-equivalence. The first equivalence of part (d) is~\cite[2.3]{Greenlees01b}, and the other follows similarly.
\end{proof}

Certain composites are immediate from the geometric descriptions.
\begin{prop}\label{prop:composite}
Suppose that $V' \subseteq V$ is an inclusion of families and
$U'\supseteq U$ an  inclusion of cofamilies. 
 For all objects $X$ of $\C$ we have natural equivalences
\begin{itemize}
\item[(a)]  $L_UL_{U'}X\xlongleftarrow{\sim}  L_U X \xlongrightarrow{\sim} L_{U'}L_U X$
\item[(b)]  $\Gamma_V\Gamma_{V'} X \xlongrightarrow{\sim} \Gamma_{V'} X \xlongleftarrow{\sim}
\Gamma_{V'}\Gamma_VX$
\item[(c)] $\Lambda_{V'}\Lambda_VX \xlongleftarrow{\sim} \Lambda_{V'} X \xlongrightarrow{\sim} \Lambda_V\Lambda_{V'} X.$ 
\end{itemize}
\end{prop}

Extreme cases also simplify.
\begin{lem}\label{lem:max and min}
Let $\C$ be a finite-dimensional Noetherian model category.
\begin{itemize}
	\item[(a)] If the Balmer spectrum of $h\C$ is irreducible 
	(i.e., there is a single generic point $\gen$) then the completion $\Lambda_{\downcl(\gen)}$ is 
	isomorphic to the identity.
	\item[(b)] If $\mathfrak{m}$ is a closed point of the Balmer spectrum, then there is a weak equivalence 
	\[\Lambda_{\downcl(\mathfrak{m})}\1 \xrightarrow{\sim} L_{\upcl(\mathfrak{m})} \Lambda_{\downcl(\mathfrak{m})} \1.\]
\end{itemize}
\end{lem}

\subsection{Splittings}
Recall from Example~\ref{eg:family dim less than} that we write $\Gamma_{\leq 0}$ and $\Lambda_{\leq 0}$ for the torsion and completion functor associated to the family of primes of dimension  zero. We use subscripts to indicate the dimension of primes. For example, $\p_0$ indicates a prime of dimension zero, and $\bigoplus_{\p_0}$ is shorthand for the sum over all primes of dimension zero. 
We note that part (a) of the following result also appears in~\cite[3.7]{Stevenson}.

\begin{prop}\label{prop-splittinf tors and compl}
For all $X \in \C$, there are equivalences \begin{itemize}
\item[(a)] $\Gamma_{\leq 0}X \xleftarrow{\sim} \bigoplus_{\p_0}\Gamma_{\downcl(\p_0)}X,$
\item[(b)] $\Lambda_{\leq 0}X \xrightarrow{\sim} \prod_{\p_0}\Lambda_{\downcl(\p_0)}X.$
\end{itemize}
\end{prop}
\begin{proof}
We prove part (a) and note that part (b) follows from applying Proposition~\ref{prop-tate cohomology functors}(c). Write $V_0$ for the family of zero dimensional primes. We must show that $\bigoplus_{\p_0}\Gamma_{\downcl(\p_0)}X$ has the universal property of the $(\K(V_0) \otimes \G)$-cellularization. In other words, we must show that $\bigoplus_{\p_0}\Gamma_{\downcl(\p_0)}X$ is in $\mathrm{Loc}_\otimes(K_{\q_0} \mid \mathrm{dim}\q_0 = 0)$ and that the natural map $\bigoplus_{\p_0}\Gamma_{\downcl(\p_0)}X \to X$ is a $(K_{\q_0} \otimes \G)$-cellular equivalence for each $\q_0$.
The first condition is clear from the fact that $\Gamma_{\downcl(\p_0)}X \in \mathrm{Loc}_\otimes(K_{\p_0})$ by definition. By compactness of $K_{\q_0} \otimes G$, the natural map $\bigoplus_{\p_0}\Gamma_{\downcl(\p_0)}X \to X$ is a $K_{\q_0} \otimes \G$-cellular equivalence if and only if $$\bigoplus_{\p_0}[K_{\q_0} \otimes G,\Gamma_{\downcl(\p_0)}X] \to [K_{\q_0} \otimes G,X]$$ is an equivalence for all zero-dimensional primes $\q_0$ and $G \in \G$. 

If $\q_0 \neq \p_0$ we claim that $[K_{\q_0} \otimes G, \Gamma_{\downcl(\p_0)}X] = 0$. Since $\Gamma_{\downcl(\q_0)}K_{\q_0} \to K_{\q_0}$ is an equivalence, $[K_{\q_0} \otimes G, \Gamma_{\downcl(\p_0)}X] = [K_{\q_0} \otimes G, \Lambda_{\downcl(\q_0)}\Gamma_{\downcl(\p_0)}X]$. The map $\Gamma_{\downcl(\p_0)}X \to 0$ is a $K_{\q_0}$-equivalence since the support of $\Gamma_{\downcl(\p_0)}K_{\q_0}$ is empty, and hence $\Lambda_{\downcl(\q_0)}\Gamma_{\downcl(\p_0)}X \simeq 0$. Therefore, \[\bigoplus_{\p_0}[K_{\q_0} \otimes G,\Gamma_{\downcl(\p_0)}X] \to [K_{\q_0} \otimes G,X]\] is an equivalence if and only if $[K_{\q_0} \otimes G,\Gamma_{\downcl(\q_0)}X] \to [K_{\q_0} \otimes G,X]$ is an equivalence, which is true by the definition of $\Gamma_{\downcl(\q_0)}$. 
\end{proof}

It is worth making explicit some consequences for the torsion part. 
\begin{cor}\leavevmode
\begin{itemize}
\item[(a)]  If the support of  $X$
consists of primes of dimension $i$  then
$$X\simeq \bigoplus_{\p_i}\Gamma_{\downcl(\p_i)}X \simeq \bigoplus_{\p_i}L_{\upcl(\p_i)}X.$$
\item[(b)]  If the support of  $X$
consists of primes of dimension $\geq i$  then 
there is an equivalence
$$\Gamma_{\leq i} X \xleftarrow{\sim}
\bigoplus_{\p_i}\Gamma_{\downcl(\p_i)}X.$$
\item[(c)]  If the support of  $X$
consists of primes of dimension $\leq i$  then 
there is an equivalence
$$L_{\geq i} X \xleftarrow{\sim} \bigoplus_{\p_i}L_{\upcl(\p_i)}X.$$ 
\end{itemize}
\end{cor}

\section{Examples}\label{sec-example}
We now illustrate the theory described in the previous sections with several examples. We will suppress notation for underlying (co)fibrant replacements in the construction of the functors $\Gamma$, $L$ and $\Lambda$, in order to not unnecessarily overburden the notation. For instance, we will often write $\Lambda X = \underline{\Hom}(\Gamma\1, X)$ where the target $X$ should be implicitly interpreted as fibrantly replaced as in Proposition~\ref{prop-tate cohomology functors}. 
\begin{nota}\label{nota:dimension}
	We emphasize that we label primes by their \emph{dimension}
        and write this as a subscript, see for instance
        Example~\ref{eg:family dim less than}. Therefore $\p_i$
        denotes a Balmer prime of dimension $i$. We also draw
        Balmer spectra with dimension increasing up the page, so that 
closed points are at the bottom. Arrows indicate inclusion.
\end{nota}

\subsection{Derived commutative algebra} \label{sec:derived commutative algebra}
Let $R$ be a Noetherian commutative ring. We write $\mod{R}$ for the category of unbounded chain complexes of $R$-modules. There is a an order reversing homeomorphism 
which is given by 
\[
 \mathrm{Spec}(R) \to \spcc(\Dsf(R)) \qquad 
 \wp \mapsto \{M \in \Dsf(R)^{\omega} \mid M_\wp \simeq 0\}.
\]
If $R$ is a finite-dimensional commutative Noetherian ring, then $\mod{R}$ is a finite-dimensional Noetherian model category when equipped with the projective model structure.

Writing $\wp$ in place of the corresponding Balmer prime, 
the functors $\Gamma_{\downcl(\wp)}$, $\Lambda_{\downcl(\wp)}$ and $L_{\upcl(\wp)}$ correspond to the usual functors of derived torsion, derived completion and localization in algebra. There are explicit models for these functors as follows. Suppose that $\wp =(x_1, \ldots, x_n)$. We define the stable Koszul complex by 
\[
K_\infty(\wp) = (R \to R[1/x_1]) \otimes_R \cdots \otimes_R(R \to R[1/x_n])
\]
and the \v{C}ech complex $\check{C}_{\wp}R $ is the suspension of the kernel of the natural map $K_\infty(\wp) \to R$. Up to equivalence this does not depend on the chosen generators of $\wp$. 
Then we have the following identifications for an $R$-module $X$:
\[
\boxed{
\Gamma_{\downcl(\wp)} X \simeq K_\infty(\wp) \otimes_R X \qquad \Lambda_{\downcl(\wp)} X \simeq
\Hom_R(K_\infty(\wp),X)\qquad L_{\upcl(\wp)} X\simeq X_\wp
}
\] 
As usual $X_\wp$ is the localization \emph{at} the ideal $\wp$ so it
inverts the multiplicative set $R\,\backslash \wp$. On the other hand,
$L_{(\downcl \wp)^c}X$ is the localization \emph{away} from $\wp$ so can be calculated as $\check{C}_{\wp}R \otimes_R X$. 
For more details see for instance~\cite{DwyerGreenlees02, GreenleesMay95b}.

For example, take $R=\Z$, $\wp=(p)$: we may take $K_{(p)} = \Z/p$. Thus
$\Gamma_{\downcl(p)}\Z \simeq \Sigma^{-1}\Z/p^\infty$, $L_{\upcl(p)}\Z \simeq \Z_{(p)}$ and
$\Lambda_{\downcl(p)} \Z \simeq \Z_p$, the $p$-adic integers.

\subsection{Chromatic homotopy theory} 
Let $\Sp_{(p)}$ denote the category of $p$-local spectra equipped with the stable model
structure.  Let $$P(n) = \{X \in h \Sp_{(p)}^{\omega}\mid K(n)_*(X) = 0\}$$
where $K(0)$ is rational homology, $K(\infty)$ is mod $p$ homology
and for finite $n\geq 1$, $K(n)$ is the $n$th mod $p$ Morava
$K$-theory. By the K\"unneth theorem the $P(n)$ are prime, and 
Hopkins--Smith~\cite{HopkinsSmith98} showed that the Balmer
spectrum is linear:
\[
P(0) \supset P(1) \supset P(2) \supset \cdots \supset P(n) \supset \cdots\supset
P(\infty). 
\]
Note that this Balmer spectrum is not finite-dimensional (and worse still, non-Noetherian). As such, we use the smashing localization $L_n=L_{K(0)\vee \cdots \vee K(n)}$ to obtain a finite-dimensional Noetherian model category with Balmer spectrum
\[
P(0) \supset P(1) \supset P(2) \supset \cdots \supset
P(n). 
\]
so that the prime $P(i)$ is of dimension $n-i$ in the $L_n$-local category.

\subsection{Rational equivariant spectra} \label{sec:rational equivariant spectra}
Let $G$ be a compact Lie group. We write $\Sp_G$ for the category of rational $G$-equivariant spectra. The second author~\cite{Greenlees19} showed that the Balmer spectrum of rational $G$-spectra is in bijection with the set of conjugacy classes of closed subgroups of $G$, with inclusion corresponding to cotoral inclusion; recall that $K \hookrightarrow H$ is \emph{cotoral} if $K$ is normal in $H$ and the quotient $H/K$ is a torus. The bijection is given by sending a conjugacy class $(H)$ to $$\p_H = \{X \mid \Phi^HX \simeq_1 \ast\}.$$
In general the Balmer
spectrum of rational $G$-spectra is not Noetherian. However, when $G$ is an $r$-torus the Balmer spectrum is
$r$-dimensional and Noetherian, so the topology is determined by the poset structure. 

For example, if $\bbT$ is the circle group, then the Balmer spectrum is given by
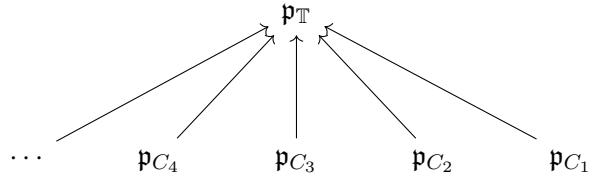
\[\begin{tikzcd}
 & & \p_{\bbT} & & \\ & & & & \\
\cdots \arrow[uurr] & \p_{C_4} \arrow[uur] & \p_{C_3} \arrow[uu] & \p_{C_2} \arrow[uul] & \p_{C_1}\arrow[uull]
\end{tikzcd}\]
where the arrows indicate cotoral inclusions.

Note that the support of a rational $G$-spectrum $X$ coincides with
its geometric isotropy $\mathcal{I}_g(X)= \{H \mid \Phi^H X \not\simeq
0 \}$. Any family $\F$ of subgroups of $G$ defines a family of
primes $\{\wp_H \mid H \in \F\}$, that by abuse of notation we still denote
by $\F$. We write $E\F$ for the universal $G$-space
characterized by its fixed points $$(E\F)^H \simeq  \begin{cases}
\ast & \mathrm{if} \; H \in \F \\
\emptyset & \mathrm{if}\; H \not\in \F
\end{cases}$$ and $\widetilde{E}\F$ for the cofibre of the nonzero map $E\F_+ \to S^0$. Then we have the following identifications
\[
\boxed{\Gamma_\F X \simeq E\F_+ \wedge X \qquad \Lambda_\F X \simeq F(E\F_+, X) \qquad 
L_{\F^c} M \simeq \widetilde{E}\F \wedge X.}
\]

\section{The standard model}\label{sec:tatesquare}

Throughout we fix a finite-dimensional Noetherian model category $\C$ and a family of primes $V$ in $\spcc(h\C)$. This gives a set of Koszul objects $\K(V)$ and associated functors $\Gamma_V, L_{V^c}, \Lambda_V$ as defined in Section~\ref{sec:functors}. When no 
confusion is likely to arise we will drop the reference to the family $V$.

\begin{rem}
 The machinery developed in this paper applies to any rigidly-compactly generated model category $\C$ equipped 
 with functors $\Gamma, L$ and $\Lambda$ satisfying the properties listed in 
 Section~\ref{sec:functors}. For example, these functors can be defined 
 whenever $\C$ is part of a local duality 
 context~\cite{BarthelHeardValenzuela18} or a Tate context~\cite{Greenlees01b}.
 We have decided to present this material from the perspective of the Balmer spectrum,
 as this approach better suits our long-term goal of constructing a 
 torsion model for a tensor-triangulated category via the dimension filtration~\cite{torsion2}. 
 \end{rem}

\begin{defn}\label{defn:T-diagram}
We write $T^\lrcorner$ for the diagram 
$$\begin{tikzcd}
& L\1 \arrow[d, "\theta"] \\ \Lambda\1 \arrow[r, "\varphi"'] & L\Lambda\1
\end{tikzcd}$$
which we shall refer to as the \emph{Tate cospan}. 
\end{defn}

\begin{lem}[{\cite[2.3]{Greenlees01b}}]\label{lem:Tate}
The homotopy pullback of $T^\lrcorner$ is weakly equivalent to $\1$.
Accordingly, for any $X$, the square $$\begin{tikzcd}
	X \arrow[r] \arrow[d] & L\1 \otimes X \arrow[d] \\
	\Lambda\1 \otimes X \arrow[r] & L\Lambda\1 \otimes X 
	\end{tikzcd}$$ is a homotopy pullback.
\end{lem}
\begin{proof} 
The fibres of $\1 \to L\1$ and $\Lambda\1 \to L\Lambda\1$ are both equivalent to $\Gamma\1$ and hence $\1$ is the homotopy pullback. 
\end{proof}

By Proposition~\ref{prop-tate cohomology functors} we know that the objects appearing in the Tate cospan are monoid objects in $\C$.  As such we can construct module categories over them.

\begin{defn}
  We define the \emph{pre-standard} category $\cpstan$ to be the category of 
  sections of the diagram category
  $$\begin{tikzcd}
& \mod{L\1} \arrow[d,"\theta_*"] \\ \mod{\Lambda\1} \arrow[r,"\phi_*"'] & \mod{L\Lambda\1}
\end{tikzcd}$$
where the arrows are the extensions of scalars coming from the ring maps $\theta\colon L\1 \to L\Lambda\1$ and $\phi\colon \Lambda\1 \to L\Lambda\1$. 
We equip this category with the \texttt{dimp} model structure where the weak equivalences and cofibrations are determined objectwise in the projective model structure on modules, see Notation~\ref{nota:dimp}.
\end{defn}

Recall from Definition~\ref{defn:leftsections} that the category $\cpstan$ has objects which are quintuples $(V,N,W,f,g)$ such that:
\begin{itemize}
\item $V \in \mod{L\1}$;
\item $N \in \mod{\Lambda\1}$;
\item $W \in \mod{L\Lambda\1}$;
\item $f\colon L\Lambda\1 \otimes_{L\1} V \to W$ in $\mod{L \Lambda \1}$;
\item $g\colon L\Lambda\1 \otimes_{\Lambda\1} N \to W$ in $\mod{L \Lambda \1}$.
\end{itemize}
We will usually draw such an object as a diagram 
\begin{equation}\label{eq-obj}
\begin{tikzcd} & V \arrow[d, dashed, "f"] \\ N \arrow[r, dashed, "g"'] & W \end{tikzcd}
\end{equation} 
where the dashed arrow indicates the existence of a map after extension of scalars. 

We define a functor $(-)_\stan \colon \C \to \cpstan$ by tensoring with the 
Tate cospan:
\[
X_\stan = \left(
\begin{tikzcd}
 & L\1 \otimes X \arrow[d, dashed]  \\
\Lambda\1 \otimes X \arrow[r, dashed] & L\Lambda\1 \otimes X
\end{tikzcd} 
\right).
\]
This functor admits a right adjoint which is given by taking the pullback. More precisely, consider an object in the pre-standard category as in Equation~(\ref{eq-obj}).
By adjunction, the data of $L\Lambda\1$-module maps $f\colon
L\Lambda\1 \otimes_{L\1} V \to W$ and $g\colon L\Lambda\1
\otimes_{\Lambda\1} N \to W$ is equivalent to the data of maps
$f^\#\colon V \to \theta^* W$ and $g^\#\colon N \to \phi^* W$ of
$L\1$-modules and $\Lambda\1$-modules respectively. Therefore, by
applying the forgetful functors, we can view the cospan $$N
\xrightarrow{g^\#} W \xleftarrow{f^\#} V$$ as a diagram in $\C$. Taking the 
pullback of this diagram gives a functor $\lim \colon \cpstan \to \C$.

\begin{thm}\label{thm:adelic}
	The adjunction \[\begin{tikzcd}\C \arrow[rr, yshift=1mm, "(-)_\stan"] & & \cell{\G_\mathrm{stan}}\cpstan\arrow[ll, yshift=-1mm, "\mathrm{lim}"]\end{tikzcd}\] is a Quillen equivalence, where $\G$ is a set of compact generators for $\C$.
\end{thm}
\begin{proof}
	The compact generators $\G$ of $\C$ get sent to compact
        objects of $\cpstan$ as in~\cite[4.1]{GreenleesShipley14b}. Since the derived unit is a weak equivalence by Lemma~\ref{lem:Tate}, the Cellularization Principle~\cite[2.7]{GreenleesShipley13} applies to give the result.
\end{proof}

The following theorem is the one-step case of the result of~\cite[9.5]{BalchinGreenlees}.
\begin{thm}\label{bifibrant in standard}
A bifibrant object $(V,N,W,f,g) \in \cpstan$ is bifibrant in $\cell{\G_\mathrm{stan}}\cpstan$ if and only if $f\colon L\Lambda\1 \otimes_{L\1} V \to W$ and $g\colon L\Lambda\1 \otimes_{\Lambda\1} N \to W$ are weak equivalences. As such, $\cell{\G_\mathrm{stan}}\cpstan$ is the strict homotopy limit of the diagram category in the sense of~\cite{bergner}.
\end{thm}

As a warm up for our approach to the torsion model in Section~\ref{sec:extended torsion model}, we observe that one can improve upon the above theorem as we now describe.

\begin{defn}
The category $\cpstane$ of \emph{extended pre-standard} objects is the full subcategory of $\cpstan$ consisting of objects $(V,N,W,f,g)$ for which the map $f \colon L \Lambda \1 \otimes_{L \1} V \to W$ is an isomorphism. Diagrammatically, the objects are those of the form
\[
\begin{tikzcd} & V \arrow[d, dashed, "\bullet" description] \\ N \arrow[r, dashed, "g"'] & W \end{tikzcd}
\]
where the dot on the arrow indicates that the map is an isomorphism after extension of scalars. By adjunction an object of $\cpstane$ is equivalently a $\Lambda \1$-module map $\beta \colon N \to L \Lambda \1 \otimes_{L \1} V$. 
\end{defn}

There is an evident inclusion $\cpstane \hookrightarrow \cpstan$. This admits a right adjoint $\Gamma_e \colon \cpstan \to \cpstane$ which is defined by $$\Gamma_e \left(\begin{tikzcd} & V \arrow[d, dashed, "f"] \\ N \arrow[r, dashed, "g"'] & W \end{tikzcd}\right) = \begin{tikzcd} & V \arrow[d, dashed, "\bullet" description] \\ N' \arrow[r] & L \Lambda \1 \otimes_{L \1} V \end{tikzcd}$$ 
where $N'$ is the pullback
\[\begin{tikzcd}
N' \arrow[r] \arrow[d] & L\Lambda\1 \otimes_{\Lambda\1} V \arrow[d, "f"] \\
N \arrow[r, "g^\#"'] & W
\end{tikzcd}
 \]
in the category of $\Lambda\1$-modules.

There exists a $\texttt{dimp}$ model structure on $\cpstane$, and a simple application of the Cellularization Principle~\cite{GreenleesShipley13} shows that $\cell{\G_\mathrm{stan}}\cpstan \simeq_Q \cell{\G_\mathrm{stan}}\cpstane$. 

\begin{defn}
The \emph{standard model} for $\C$ is the model category  $\C_s = \cell{\G_\mathrm{stan}}\cpstane$.
\end{defn}

Combining this with the previous results of this section we obtain the following.

\begin{thm}
There is a zig-zag of Quillen equivalences
\[\C \simeq_Q \C_s. \]
Moreover an object in the homotopy category of $\C_s$ is determined by the data of a $\Lambda \1$-module map $\beta\colon N \to L\Lambda\1 \otimes_{L\1} V$ such that $L\Lambda\1 \otimes_{\Lambda\1}\beta$ is an equivalence, where $N$ is a $\Lambda\1$-module and $V$ is an $L\1$-module.
\end{thm}

\section{The pre-torsion model}
We now move towards the torsion model. We note that the setup remains the same as in Section~\ref{sec:tatesquare}. Implicitly we have fixed a family of 
Balmer primes $V$ and a set of Koszul objects $\K=\K(V)$. 

\subsection{Enriching the splicing maps}
In algebra, any torsion module is naturally a module over the adic completion. This is the crucial difference between the adelic and Zariski models. 
The following lemma gives a version for our context. 

\begin{lem}\label{lem:torsion over complete ring}
There is a Quillen equivalence $\cell{(\K \otimes \G)}\C \simeq_Q \cell{(\Lambda \1 \otimes \K \otimes \G)}\mod{\Lambda \1}(\C).$ 
\end{lem}
\begin{proof}
There is a ring map $\theta\colon \1 \to \Lambda\1$ which induces a Quillen adjunction $$\Lambda\1 \otimes -: \C \rightleftarrows \mod{\Lambda\1}: \theta^*.$$ We apply the Cellularization Principle~\cite[2.7]{GreenleesShipley13} at the set of compact objects $\K \otimes \G$ in $\C$. The objects $\Lambda\1 \otimes K \otimes G$ are compact as $\Lambda\1$-modules for all $K \in \K$ and $G \in \G$ since the restriction of scalars functor preserves sums. The derived unit $K \otimes G \to \Lambda\1 \otimes K \otimes G$ is an equivalence for all $K \in \K$ and $G \in \G$ by the MGM equivalence,  since each $K$ is torsion and $\Gamma$ is smashing. Therefore the stated Quillen equivalence holds.
\end{proof}

 In light of the previous
 lemma, when considering torsion modules we can work over
 $\Lambda\1$. 

\begin{defn}\label{def:pt}
We define the \emph{pre-torsion} category $\cptors$  to be the category of  sections of the diagram category
$$
\begin{tikzcd}
	\mod{L\1} \arrow[d, "\theta_*"'] &  \\ 
	\mod{L\Lambda \1} \arrow[r, "\phi^*"'] &  \mod{\Lambda \1}.
\end{tikzcd}
$$
The vertical map $\theta_*$ is the extension of scalars $L \Lambda\1 \otimes_{L\1} -$ along the ring map $\theta\colon L\1 \to L\Lambda\1$, while the horizontal map is restriction of scalars along the ring map $\phi\colon\Lambda\1 \to L\Lambda\1$.
\end{defn}

\begin{rem}
  The categories $\cpstan$ and $\cptors$ are defined as the categories 
  of sections over two very similar diagrams. 
  The crucial difference lies in the map 
  $\varphi\colon\Lambda\1\to L\Lambda\1$: in the first case we extend 
  scalars along $\varphi$ whereas in the second we restrict scalars along 
  $\varphi$. 
\end{rem}

By Definition~\ref{defn:leftsections} an object in $\cptors$ is a quintuple $(V,W,T,f,g)$ where:
\begin{itemize}
\item  $V \in \mod{L\1}$;
\item $W \in \mod{L\Lambda\1}$;
\item $T \in \mod{\Lambda\1}$;
\item $f\colon \theta_\ast V \to W$ in $\mod{L \Lambda \1}$;
\item $g\colon \phi^*W \to T$ in $\mod{\Lambda \1}$.
\end{itemize}
The morphisms in $\cptors$ are triples $(\alpha\colon V \to V', \beta\colon W \to W', \gamma\colon T \to T')$ such that the evident squares commute. 

As in the pre-standard model, we will often denote this object by
\begin{center}
	\begin{tikzcd}
	V \arrow[d, dashed, "f"'] & \\ W \arrow[r, "g"'] & T
	\end{tikzcd}
\end{center}
where we have not made the horizontal morphism dashed as we will rarely make any reference to the restriction of scalars.

In order to equip $\cptors$ with a model structure, we require both $\theta_\ast$ and $\phi^\ast$ to both be Quillen of the same handedness. In general, $\theta_\ast$ will not be a right adjoint, and so we view both the functors as left adjoints.

Recall from Lemma~\ref{lem:restrictionleftQuillen} that in the projective model structure on modules, $\theta_\ast$ is always left Quillen, whereas $\phi^\ast$ is left Quillen if and only if $L\Lambda\1$ is cofibrant as an $\Lambda\1$-module. We now describe how to pick models for $L\1$, $\Lambda\1$ and $L\Lambda\1$ with this property. This will allow us to endow $\cptors$ with the \texttt{dimp} model structure from Notation~\ref{nota:dimp}. There are two approaches depending on whether the commutativity of the rings is important. 

\subsection{Non-commutative case}\label{sec:nc}
We may choose a model for $L\1$ which is cofibrant as a ring, and therefore as an object of $\C$ by~\cite[4.1(3)]{SchwedeShipley00}. Tensoring the map $\1 \to L\1$ with $\Lambda\1$ gives a ring map $\Lambda\1 \to L\1 \otimes \Lambda\1$ which is a cofibration of $\Lambda\1$-modules. As $\Lambda\1$ is a cofibrant $\Lambda\1$-module, this shows that $L\1 \otimes \Lambda\1$ is a cofibrant $\Lambda\1$-module. Since $L$ is smashing, $L\1 \otimes \Lambda\1$ is a model for $L\Lambda\1$, and since it is a cofibrant $\Lambda\1$-module, the restriction of scalars functor $\phi^*\colon \mod{L\Lambda\1} \to \mod{\Lambda\1}$ is left Quillen as desired.

We emphasise that since we took cofibrant replacements as rings/algebras, the resulting objects need
not admit a commutative multiplication. However, we give an alternative approach in the next subsection. 

\subsection{Commutative case}\label{sec:convenient model structure}
In many cases of interest, the rings at play in the torsion model are
actually \emph{commutative}. For example, if $\bbT$ is the circle group, and $\F$ denotes the family of finite subgroups of $\mathbb{T}$, then $\widetilde{E}\F$ is a (genuinely) commutative ring $\mathbb{T}$-spectrum~\cite{Greenlees20b}. If $L\1$ and $\Lambda\1$ are commutative, and the category supports a \emph{convenient} model structure, then we can preserve commutativity in the torsion model.

\begin{defn}\label{defn:convenientmodelstructure}
	Let $\C$ be a monoidal model category and suppose that there exists another monoidal model structure $\widetilde{\C}$ on the same underlying category as $\C$, which has the same weak equivalences and for which the identity functor $\widetilde{\C} \to \C$ is left Quillen. We say that $(\C, \widetilde{\C})$ is \emph{convenient} if the forgetful functor $\calg{S}(\widetilde{\C}) \to \mod{S}(\C)$ preserves cofibrant objects, for all commutative monoids $S \in \C$. 
\end{defn}

\begin{exs}
The following pairs of model structures are all convenient:
\begin{enumerate}
\item (projective, projective) on chain complexes over a field of characteristic zero, see for instance~\cite[\S 5.1]{White17}.
\item (stable, positive stable) model structure on spectra is \emph{not} convenient. However, Shipley~\cite[4.1]{Shipley04} constructs a flat and positive flat model structure on spectra so that (flat, positive flat) is convenient. We note that there is a left Quillen equivalence from the stable model structure to the flat model structure.
\item The previous example also works for equivariant spectra. More
  precisely, the pair of model structures (flat, positive flat)  on
  equivariant spectra constructed in~\cite[4.2.15]{BrunDundasStolz}
  is convenient.
\end{enumerate}
\end{exs}

In order to ensure that the restriction of scalars functor $\phi^*$ is left Quillen, we first take a model of $L\1$ which is cofibrant as a commutative ring in the model structure right lifted from $\widetilde{\C}$. If $(\C, \widetilde{\C})$ is convenient then this model of $L\1$ is also cofibrant as an object of $\C$. Therefore the ring map $\Lambda\1 \to L\1 \otimes \Lambda\1$ is a cofibration of $\Lambda\1$-modules, and as such the restriction of scalars along this map is left Quillen.


\subsection{Constructing the pre-torsion model} We now define an adjunction between $\C$ and $\cptors$ and show that after a suitable cellularization this gives a Quillen equivalence. This is the first step towards the torsion model constructed in the next section.


We define a functor $(-)_\mathrm{tors}\colon \C \to \cptors$ by \[X_\mathrm{tors} = \left(\begin{tikzcd}
L\1 \otimes X \arrow[d, dashed] & \\
L\Lambda\1 \otimes X \arrow[r] & \Sigma \Gamma\1 \otimes \Lambda \1 \otimes X
\end{tikzcd} \right)\]
where the vertical map is induced by the identity and the horizontal map is induced by 
the triangle $\Lambda \1 \to L\Lambda \1 \to \Sigma \Gamma\Lambda\1$.

\begin{rem}\label{rem:abuse}
In light of Lemma~\ref{lem:torsion over complete ring} and the MGM equivalence $\Gamma \simeq \Gamma\Lambda$, any torsion module is equivalent to a $\Lambda\1$-module. As such, the object $\Sigma \Gamma \1 \otimes \Lambda \1 \otimes X$ is equivalent to $\Sigma \Gamma X$.
\end{rem}

The functor $(-)_\mathrm{tors}\colon \C \to \cptors$ has a right adjoint $\mathsf{R}\colon \cptors \to \C$ defined by sending the object \[\begin{tikzcd}V \arrow[d, dashed, "f"'] & \\ W \arrow[r, "g"'] & T\end{tikzcd}\] to the pullback of the cospan \[\begin{tikzcd}
 & V \arrow[d, "f^\#"] \\
 \mathrm{fib}(g\colon W \to T) \arrow[r] & W.
\end{tikzcd}\]
Note that $f\colon L\Lambda\1 \otimes_{L\1} V \to W$ is adjunct of a map $f^\#\colon V \to W$ of $L\1$-modules. Therefore, applying the forgetful functors we view the cospan as a diagram in $\C$ and take the pullback in $\C$. 

\begin{rem}\label{rem-derived-equivalence}
Equivalently, the functor $\mathsf{R}$ is the fibre of the composite $g f^\#$. 
Using this description one sees that the derived unit $X \to \mathsf{R}(X_\tors)$ is an equivalence for all $X \in \C$ since the fibre of the canonical map $L\1 \otimes X \to \Sigma \Gamma \1 \otimes \Lambda \1 \otimes X$ is $X$ by the MGM equivalence. 
\end{rem}

\begin{lem}\label{lem:1torssmall}
The objects in $\G_\mathrm{tors}$ are compact in the homotopy category of $\cptors$. 
\end{lem}

\begin{proof}
The evaluation functors $\cptors \to \mod{L\1}$ and $\cptors \to \mod{\Lambda\1}$ have left adjoints defined by 
\[a(V)=\left(\begin{tikzcd} V \arrow[d, dashed] & \\ L\Lambda\1 \otimes_{L\1} V \arrow[r] & L\Lambda\1 \otimes_{L\1} V \end{tikzcd}\right)
\mbox{ and } b(T)=\left(\begin{tikzcd} 0 \arrow[d, dashed] & \\ 0 \arrow[r] & T \end{tikzcd}\right)\] respectively. It is clear from the definitions of $a$ and $b$ that they preserve cofibrations and weak equivalences since these are defined objectwise. Therefore $a$ and $b$ are left Quillen. Note that the derived functors of $a$ and $b$ preserve compact objects since the derived evaluation preserves sums. Since cofibre sequences are vertexwise, for all $G \in \G$ we have a cofibre sequence
$$b(\Lambda \1 \otimes G)\to a(L\1\otimes G)\to G_\mathrm{tors}$$ 
induced by the triangle $\Lambda \1 \to L\Lambda \1\to \Sigma \Gamma\Lambda\1$ on the 
right bottom vertex. This shows that $G_\mathrm{tors}$ is compact in the homotopy category of $\cptors$.
\end{proof}

We can now prove the main theorem of this section. 
\begin{thm}\label{thm:pretorsion} 
	The adjunction $(-)_\mathrm{tors}: \C \rightleftarrows \cell{\G_\mathrm{tors}}\cptors:\mathsf{R}$ is a Quillen equivalence.
\end{thm}
\begin{proof}
We first verify that this is a Quillen adjunction and then apply the Cellularization Principle~\cite{GreenleesShipley13} at the set $\G$ of compact generators of $\C$. Recall that $\cptors$ is equipped with the \texttt{dimp} model structure, so we must check that $(-)_\mathrm{tors}$ sends cofibrations (resp., acyclic cofibrations) in $\C$ to objectwise cofibrations (resp., acyclic cofibrations). It is clear for the $L\1$ and $L\Lambda\1$ coordinates since extension of scalars is left Quillen. In the final coordinate $(-)_\mathrm{tors}$ is given by tensoring with $\Sigma\Gamma\1 \otimes \Lambda\1 \otimes -$. Therefore we must show that $\Sigma\Gamma\1 \otimes \Lambda\1 \otimes -$ sends cofibrations (resp., acyclic cofibrations) in $\C$ to cofibrations (resp., acyclic cofibrations) of $\Lambda\1$-modules. As extension of scalars along $\1 \to \Lambda\1$ is left Quillen, it suffices to show that $\Gamma \1 \otimes -$ is left Quillen as an endofunctor of $\C$. This is true since $\Gamma\1$ is cofibrant in $\C$ by definition of the functor $\Gamma$. Therefore, $(-)_\mathrm{tors}: \C \rightleftarrows \cptors : \mathsf{R}$ is a Quillen adjunction.

By Lemma~\ref{lem:1torssmall}, $\G_\mathrm{tors}$ is a set of compact objects of the homotopy category of $\cptors$. The derived unit of the Quillen adjunction \[(-)_\mathrm{tors}: \C \rightleftarrows \cptors : \mathsf{R}\] is a weak equivalence by 
Remark~\ref{rem-derived-equivalence}. Therefore the Cellularization Principle shows that the stated Quillen equivalence holds.
\end{proof}

\subsection{Relation to the pre-standard model}\label{sec:adelic vs torsion}
The adjunction $(-)_\mathrm{tors}: \C \rightleftarrows \cptors : \mathsf{R}$ factors through the pre-standard category $\cpstan$ as follows. We define a functor $\mathrm{hcof}\colon \cpstan \to \cptors$ by $$\mathrm{hcof} \left(\begin{tikzcd} & V \arrow[d, dashed] \\ N \arrow[r, dashed] & W \end{tikzcd}\right) = \begin{tikzcd} V \arrow[d, dashed] & \\ W \arrow[r] 
	& \mathrm{cof}(N \to W)\end{tikzcd}$$ by taking the horizontal cofibre in the category of $\Lambda\1$-modules. Note that by adjunction, the map $L\Lambda\1 \otimes_{\Lambda\1} N \to W$ defines a map of $\Lambda\1$-modules $N \to W$, and it is this map whose horizontal cofibre we take.
Similarly, one defines a functor $\mathrm{hfib}\colon \cptors \to \cpstan$ by taking the horizontal fibre as $\Lambda\1$-modules. One notes that the left and right adjoint of the composite adjunction
	\[
	\begin{tikzcd}
		\C \arrow[r, yshift=1mm, "(-)_\stan"] & \cpstan \arrow[r, "\mathrm{hcof}", yshift=1mm] \arrow[l, yshift=-1mm, "\mathrm{lim}"] & \cptors. \arrow[l, yshift=-1mm, "\mathrm{hfib}"]
	\end{tikzcd}\] 
are $(-)_\mathrm{tors}=\mathrm{(hcof)}\circ (-)_\stan$ and
$\mathsf{R} =\mathrm{(lim)}\circ \mathrm{(hfib)}$.

It is also easy to see that these functors give a Quillen equivalence $\mathrm{hcof}\colon \cpstan \rightleftarrows \cptors:\mathrm{hfib}$ and hence a Quillen equivalence \[\mathrm{hcof}\colon \cell{ \G_\stan}\cpstan \rightleftarrows \cell{\G_\mathrm{tors}}\cptors:\mathrm{hfib}\] relating the models for $\C$ described in Theorems~\ref{thm:adelic} and~\ref{thm:pretorsion}.

\section{The torsion model}\label{sec:extended torsion model}
In this section we provide a Quillen equivalence between the model
category $\cell{\G_\mathrm{tors}}\cptors$ considered in the previous
section and a cellularization of a certain subcategory $\cptorse$. The
cellularization  of $\cptorse$ is the \emph{torsion model} $\ctors$. This description will allow us to understand the homotopy category of the torsion model.

\begin{defn} 
The category  $\cptorse$ of \emph{extended pre-torsion} objects is the full subcategory of $\cptors$ consisting of objects $(V,W,T,f,g)$ for which the map $f\colon L\Lambda\1 \otimes_{L\1} V \to W$ is an isomorphism. Diagramatically, the objects are those of the form
\[\begin{tikzcd}
V \arrow[d, dashed, "\bullet" description] & \\ W \arrow[r, "g"']
& T 
\end{tikzcd}\]
where the dot on the arrow indicates that the map is an isomorphism after extension of scalars. For brevity, and in line with \cite{Greenlees99}, we will often say that an object of the extended pre-torsion model is a $\Lambda\1$-module map $\beta\colon L\Lambda \1 \otimes_{L\1} V \to T$. In this notation a map is a commutative square 
\[
\begin{tikzcd}
L\Lambda\1 \otimes_{L\1} V \arrow[d,"L \Lambda\1\otimes_{L\1} \varphi"'] 
\arrow[r,"\beta"] & T \arrow[d, "\psi"] \\
L \Lambda\1 \otimes_{L\1} W \arrow[r,"\gamma"'] & U.
\end{tikzcd}
\]
\end{defn}

There is an adjunction 
\[ \begin{tikzcd}\cptors \arrow[r, "\Gamma_e"', yshift=-1mm] & \cptorse \arrow[l, "i"', yshift=1mm]\end{tikzcd}\]
where $i$ is the evident inclusion and $\Gamma_e$ is defined by 
$$\Gamma_e \left(\begin{tikzcd} V \arrow[d, dashed, "f"'] &  \\ W \arrow[r, "g"'] & T \end{tikzcd}\right) = \begin{tikzcd} V \arrow[d, dashed, "\bullet" description] &  \\ L \Lambda \1 \otimes_{L \1} V \arrow[r, "gf"'] & T. \end{tikzcd}$$
We note that $\cptorse$ has all limits and colimits. Colimits are calculated in $\cptors$ 
whereas limits are obtained by applying the functor $\Gamma_e$ to the limits in $\cptors$.
By~\cite[2.1]{Haraguchi15}, we may endow the category $\cptorse$ with the $\texttt{dimp}$ model structure, see Notation~\ref{nota:dimp}. 

\begin{lem}\label{lem-extended-torsion-model}
There is a Quillen equivalence
\[\begin{tikzcd}\cell{\G_\mathrm{tors}}\cptors \arrow[r, "\Gamma_e"', yshift=-1mm] & \cell{\G_\mathrm{tors}}\cptorse. \arrow[l, "i"', yshift=1mm]\end{tikzcd}\]
\end{lem}
\begin{proof}
Since the weak equivalences and cofibrations are determined objectwise in the diagram injective model structure, it is clear that the inclusion $i$ preserves them. Therefore, the adjunction (without the cellularizations) is a Quillen adjunction. Since $i$ sends $\G_\mathrm{tors}$ to $\G_\mathrm{tors}$, the Quillen adjunction passes to the cellularizations.
By Lemma~\ref{lem:1torssmall}, $\G_\mathrm{tors}$ is a set of compact objects in $\cptors$, and it follows similarly that the objects of $\G_\mathrm{tors}$ are also compact in $\cptorse$. Therefore by the Cellularization Principle~\cite{GreenleesShipley13} it is enough to show that the derived unit is a weak equivalence on $\G_\mathrm{tors}$ which is clear.
\end{proof}

\begin{defn}
  The \emph{torsion model} of $\C$ is the model category 
  \[
  \C_t=\cell{\G_\mathrm{tors}}\cptorse.
  \]
\end{defn}

Combining the results of Theorem~\ref{thm:pretorsion} and Lemma~\ref{lem-extended-torsion-model} gives the following statement.
\begin{thm}\label{thm:torsionmodel}
There is a zig-zag of Quillen equivalences $\C \simeq_Q \C_t$.
\end{thm}

We now turn our attention to identifying the homotopy category of the torsion model $\C_t$. Throughout we will use the following observation.

\begin{rem} 
If $\L$ is a set of compact objects, then the homotopy category of $\cell{\L}\C$ can be identified with the localizing subcategory of $h\C$ generated by $\L$, see~\cite[2.5, 2.6]{GreenleesShipley13}.
\end{rem} 

\begin{lem}\label{cor:gamma and K}
An object $X \in \C$ is $(\Gamma\1\otimes \G)$-cellularly acyclic  if and only if it is $(\K \otimes \G)$-cellularly acyclic. Therefore, the cellularizations of $\C$ at the sets $(\Gamma \1 \otimes \G)$ and $(\K \otimes \G)$ are equal as model categories.

\end{lem}
\begin{proof}
Since $\Gamma\1$ is torsion the reverse implication is clear. Conversely suppose that $[\Gamma\1 \otimes G, X] \simeq 0$ for all $G \in \G$. This means that $[G, \underline{\Hom}(\Gamma \1, X)]=0$ for all $G \in \G$ and so $\Hom(\Gamma \1, X) \simeq 0$ in $\C$. Then $[K \otimes G, \underline{\Hom}(\Gamma\1, X)] \simeq 0$ for each $K \in \K$, and since $\Gamma K \simeq K$ we conclude that 
$[K \otimes G, X]=0$ as required. The second claim then follows since the cellularly acyclic objects determine the weak equivalences in the cellularization. \end{proof}

\begin{thm}\label{thm:bifibrant}
A bifibrant object $(\beta\colon L\Lambda \1 \otimes_{L\1} V \to T)$ of $\cptorse$ is bifibrant in $\C_t$ if and only if the natural map 
$\Gamma T \to T$ is a weak equivalence in $\C$.
\end{thm}
\begin{proof}
Let $\L=\L(\K,\G)$ be the set of objects of $\cptorse$ of the form 
\[
(0 \to \Lambda \1 \otimes K \otimes G)\quad \mathrm{and}\quad (L\Lambda \1 \otimes G \to 0)
\]
where $K$ ranges over $\K$ and $G$ ranges over the generators $\G$. 
Firstly, we show that $\C_t =\cell{\mathscr{\L}}\cptorse$ as model categories. 
Note that there is a triangle
\[
(0 \to \Sigma \Gamma \1 \otimes \Lambda \1 \otimes G) \to (L\Lambda\1 \otimes G \to \Sigma\Gamma \1 \otimes \Lambda \1 \otimes G) \to (L\Lambda\1 \otimes G \to 0)
\]
and combining this with Lemma~\ref{cor:gamma and K} shows that an
object of $\cptorse$ is $\G_\mathrm{tors}$-cellularly acyclic if and only
if it is $\L$-cellularly acyclic. Therefore the
$\G_\mathrm{tors}$-cellularization and the $\L$-cellularization give the
same model category.

We now claim that the objects in $\L$ are compact. Since $\Lambda \1 \otimes K \otimes G$ is compact as a $\Lambda 1$-module and sums in $\cptorse$ are computed levelwise, we immediately see that $(0 \to \Lambda \1 \otimes K \otimes G) \in \L$ is compact. To see that $(L\Lambda\1 \otimes G \to 0)$ is compact, we show that any map in $\cptorse$
\[
\begin{tikzcd}
L\Lambda\1 \otimes_{L\1}(L\1\otimes G) \arrow[d, "L\Lambda \1 \otimes \varphi"'] \arrow[r ]& 0 \arrow[d] \\
L \Lambda\1 \otimes_{L{\1}}( \bigoplus_{i} V_i)  \arrow[r] & \bigoplus_{i}T_i
\end{tikzcd}
\]
factors through a finite stage. This is clear since the image of $\varphi$ is nonzero only for finitely many indices $i$ as $L\1 \otimes G$ is a compact 
$L\1$-module.

Since the objects in $\L$ are compact, the bifibrant objects in $\cell{\L}\cptorse$ are exactly the objects of the localizing subcategory $\mathrm{Loc}(\L).$ From this it is clear that $(\beta\colon L\Lambda \1 \otimes_{L\1} V \to T) \in \cptorse$ is bifibrant in the cellularization if and only if $\Gamma T \to T$ is an equivalence.
\end{proof}

\begin{rem}\label{rem-hmtpy-cat-e-tors-model}  By Theorem \ref{thm:bifibrant}, an object in the homotopy category of the torsion model 
  $\C_t$ is represented by the following data:
  \begin{itemize}
  \item[(i)] a $L\1$-module $V$;
  \item[(ii)] a $\Lambda \1$-module $T$ which is torsion in the sense that the 
  natural map  $\Gamma T \to T$ is a weak equivalence;
  \item[(iii)] a $\Lambda\1$-module map 
  $\beta \colon L\Lambda\1 \otimes_{L\1} V \to T$.
  \end{itemize}

Since extension of scalars along $\1\to \Lambda \1$ does not affect
cellularization by Lemma~\ref{lem:torsion over complete ring}, we may assume $T$ is a $\Lambda \1$-module and can replace (ii) and (iii) 
  with 
  \begin{itemize}
  \item[(ii')] an object $T \in \C$ such that $\Gamma T \to T$ is a weak 
  equivalence;
  \item[(iii')] a map $\beta' \colon V \to T$ in $\C$.
  \end{itemize}

We call (i), (ii), (iii)  \emph{adelic data} and (i), (ii)', (iii)' 
\emph{Zariski data}.  The message of this paper is that one should use
the adelic data because they are closer to algebra than the Zariski
data, and we think of the adelic data as an enrichment of the Zariski
data. For example with rational $SO(2)$-spectra, in the Zariski formulation we are working with
modules over the $SO(2)$-equivariant rational sphere which has a very
complicated algebraic structure. However, in the adelic formulation we
work over $DE\F_+$, which is formal and $\pi_*^\bbT(DE\F_+)=\cO_\F$ is
a product of polynomial rings. See Section~\ref{sec:Tspectra} for more details.
\end{rem}

\begin{ex}
  Let us consider the torsion model for $\mod{\mathbb{Z}}$ with the
  family of closed points, i.e., the family of primes corresponding to the algebraic ideals $(p)$ for $p>0$ prime. Recall that a maximal 
  prime in $\mathrm{Spec}(\mathbb{Z})$ becomes minimal in the Balmer spectrum. The set $\{\mathbb{Z}/p \mid \text{$p$ prime}\}$ is a set of Koszul objects for this family of primes.
  Using Proposition~\ref{prop-splittinf tors and compl} and 
  Section~\ref{sec:derived commutative algebra} we calculate:
  \[
  L\1= \mathbb{Q} \qquad 
  \Gamma\1 =\bigoplus_{p \;\mathrm{prime}} \Sigma^{-1}\mathbb{Z}/p^\infty \qquad 
  \Lambda\1= \prod_{p \; \mathrm{prime}} \mathbb{Z}_p.
  \]
  Furthermore by Proposition~\ref{prop-splittinf tors and compl}, a complex $T \in \mod{\mathbb{Z}}$ satisfies 
  $\Gamma T \simeq T$ if and only if it decomposes as $T \simeq \bigoplus_p T_p$ where 
  $T_p$ is derived $p$-torsion.
  Therefore the data in the 
  Zariski torsion model for $\mod{\mathbb{Z}}$ consists of: 
  \begin{itemize}
  \item a rational chain complex $V \in \mod{\mathbb{Q}}$;
  \item a derived $p$-torsion complex $T_p$ for each prime 
  number $p$;
  \item a chain map $\beta \colon V \to \bigoplus_p T_p$. 
  \end{itemize}
  The adelic torsion model instead consists of:
  \begin{itemize}
  \item a rational chain complex $V \in \mod{\mathbb{Q}}$;
  \item a derived $p$-torsion complex $T_p$ for each prime 
  number $p$ but now seen as a module over $\Z_p$;
  \item a map 
  $(\prod_p \Z_p) \otimes V \to \bigoplus_p T_p$ of $(\prod_p \Z_p)$-modules.
  \end{itemize}
\end{ex}

\begin{rem}
Finally let us record the precise relationship between the standard model and the torsion model. Combining the results of Section~\ref{sec:adelic vs torsion} with Lemma~\ref{lem-extended-torsion-model}, one sees that there is a zig-zag of Quillen equivalences
\[
\begin{tikzcd}
\C_s \arrow[r, hookrightarrow, yshift=1mm] & \cell{\G_\mathrm{stan}}\cpstan \arrow[r, yshift=1mm, "\mathrm{hcof}"]  \arrow[l, yshift=-1mm, "\Gamma_e"] & \cell{\G_\mathrm{tors}}\cptors \arrow[r, yshift=-1mm, "\Gamma_e"'] \arrow[l, yshift=-1mm, "\mathrm{hfib}"] & \C_t \arrow[l, yshift=1mm, hookrightarrow] 
\end{tikzcd}
 \]
 between the standard model and torsion model for $\C$.
\end{rem}

\section{Cellular skeleton theorems}\label{sec:cellular skeleton thm}
The cellularization of the model structure present in the torsion model forces one of the objects to be derived torsion, see Theorem~\ref{thm:bifibrant}. If one is in a suitably algebraic situation, one can internalize this cellularization and prove a Quillen equivalence to a category of differential objects in an abelian category where this object is now torsion at the abelian level. We call such a Quillen equivalence a \emph{cellular skeleton theorem}, and in this section we prove such a result. We then apply this to simplify the torsion model for the derived category of a regular local ring of Krull dimension 1. In  Section~\ref{sec:Tspectra} the cellular skeleton theorem will provide the final step in proving the algebraic torsion model for rational $\mathbb{T}$-spectra.

\subsection{Torsion functors}
In this section we fix a commutative ring $R$ and a multiplicatively closed set $\cE$ of $R$.
\begin{defn}
The $\cE$-torsion functor $\Gamma_{\cE}\colon \mod{R} \to \mod{R}$ is defined by sending an $R$-module $N$ to the kernel of the map $N \to \cE^{-1}N$. If the natural map $\Gamma_{\cE}N \to N$ is an isomorphism, we say that $N$ is \emph{$\cE$-torsion.} We write $\mod{R}^{\text{$\cE$-tors}}$ for the full subcategory of $\mod{R}$ consisting of the $\cE$-torsion modules.
\end{defn}

\begin{rem}
The $\cE$-torsion functor is easily expressed in terms of the torsion functors for principal ideals $$\Gamma_{\cE} N=\mathrm{colim}_{a\in \cE}\Gamma_aN.$$
\end{rem}

  We will be working under the following additional assumption.

\begin{hyp}\label{hyp-torsion preserves injective}
   The torsion functor $\Gamma_{\cE} \colon \mod{R} \to \mod{R}$ preserves injective modules.
\end{hyp}

 If $R$ is Noetherian then this hypothesis holds as a consequence of the Baer criterion, see~\cite[2.1.4]{Brodmann-Sharp}.
One immediately sees that $\mod{R}^{\text{$\cE$-tors}}$ is an abelian 
 category with enough injectives. This hypothesis ensures that the injective 
 dimension of a torsion module is the same whether we think of it in
  the category of $\cE$-torsion modules or in the category of $R$-modules.

\subsection{The cellular skeleton theorem}
Suppose we have a diagram of commutative rings $S \to \cE^{-1}R \leftarrow R$ where $\cE$ is a multiplicative set in $R$. Furthermore, we assume that the $\cE$-torsion functor preserves injective modules as in Hypothesis~\ref{hyp-torsion preserves injective}.

Consider the sections of the diagram category 
\[\begin{tikzcd}
\mod{S} \arrow[d, "\bullet" description] & \\ \mod{\cE^{-1}R} \arrow[r]
& \mod{R} 
\end{tikzcd}\]
where the vertical is extension of scalars along the ring map $S \to \cE^{-1}R$ and the horizontal is restriction of scalars along the ring map $R \to \cE^{-1}R$.  We denote the sections of this category by $\widehat{\A}_t$.  An object of $\widehat{\A}_t$  is given by the data of an $R$-module map $\cE^{-1}R \otimes_S V \to N$ where $V$ is an $S$-module and $N$ is an $R$-module.

We are interested in situations when $\widehat{\A}_t$ is cellularized in such a way that the map $\Gamma_{\cE} N \to N$ is forced to be a quasi-isomorphism, so that in particular, the homology of $N$ is $\cE$-torsion. The goal of this section is to prove that one can strictify this at the abelian level by replacing $N$ with a $\cE$-torsion module. 

\begin{defn}
We define the category $\A_t$ to have objects $R$-module maps $\cE^{-1}R \otimes_S V \to T$ where $V$ is an $S$-module and $T$ is a $\cE$-torsion $R$-module. Morphisms in $\A_t$ are commutative squares
\[
\begin{tikzcd}
 \cE^{-1} R \otimes_S V \arrow[r] \arrow[d, "1 \otimes \varphi"'] & 
 T \arrow[d, "\psi"] \\
 \cE^{-1} R \otimes_S W \arrow[r] & U.
\end{tikzcd}
\]
The category $\A_t$ is abelian with kernels and cokernels defined levelwise.
\end{defn}

\begin{lem}\label{lem:abelian}
\leavevmode
\begin{itemize}
\item[(a)] The abelian category $\A_t$ has enough injectives.
\item[(b)] The category $d\A_t$ of differential objects in $\A_t$
  supports a \texttt{dimi} model structure in which the weak
  equivalences are the objectwise quasi-isomorphisms and the
  cofibrations are the objectwise monomorphisms. 
\item[(c)] Suppose that $S$ is a field. The injective dimension of $\A_t$ is equal to $\mathrm{injdim}(\mod{R}) + 1.$
\end{itemize}
\end{lem}
\begin{proof}
 For an $S$-module $V$ and an $\cE$-torsion $R$-module $T$, we put 
 \[
 e^t(V)=(\cE^{-1}R \otimes_S V \to 0) \quad \mathrm{and} \quad 
 f^t(T)= (\cE^{-1} R\otimes_S \Hom_R(\cE^{-1}R, T) \to T) 
 \]
 where the second map is the evaluation. A straightforward calculation shows that 
 $e^t$ and $f^t$ are right adjoint to the respective forgetful functors.  
 The forgetful functors are exact as (co)kernels in $\A_t$ are defined 
 levelwise, and therefore $e^t$ and $f^t$ preserve injectives. 
 Consider an object $X=(\cEi R \otimes_S V \to T)\in \A_t$ and 
 choose a monomorphism $T \to I$ where $I$ is an injective $\cE$-torsion 
 module. Then there is a monomorphism $X\to f^t(I)\oplus e^t(V)$ and so 
 part (a) follows. Part (b) then follows from~\cite[Appendix B]{Greenlees99}.
 
For part (c), we write $r$ for the injective 
 dimension of the category of $R$-modules. 
 Consider an object $X=(\cEi R \otimes_S V \to T)$ of $\A_t$ and choose an
 injective resolution 
 \[
 0 \to T \xrightarrow{\alpha_0} I_0 \xrightarrow{\alpha_1} I_1 \to \ldots \to 
 I_{r-1} \xrightarrow{\alpha_r} I_r \to 0
 \] 
 of $T$. We put $V_i = \Hom_R(\cEi R, I_i)$
 for all $i\geq 0$. Then one checks that that there is an exact 
 sequence
 \[
 0 \to X \to e^t(V) \oplus f^t(I_0) \to e^t(V_0)\oplus f^t(I_1) \to \ldots 
 \to e^t(V_{r-1}) \oplus f^t(I_r) \to e^t(V_r) \to 0
 \]
 which provides an injective resolution of $X$ in $\A_t$. This shows that the injective dimension 
 is less than or equal to $r+1$. Finally we note that the bound can be 
 achieved. For example, using the notation above
 \[
 \mathrm{Ext}_{\A_t}^{r+1}(e^t(S), (0 \to T)) = 
 \Hom_{\A_t}(e^t(S), e^t(V_r))=V_r= \mathrm{Ext}_R^r(\cEi R, T)
 \]
 which need not be zero. 
 \end{proof}

\begin{rem}
The category $\widehat{\A}_t$ supports both a \texttt{dimi} and \texttt{dimp} model structure, but the category $d\A_t$ supports only a \texttt{dimi} model structure since there are enough torsion injectives, but in general no torsion projectives. 
\end{rem}

We now construct an adjunction $(i, \Gamma_t)$ between the categories $\widehat{\A}_t$ and $d\A_t$ described above. There is an evident inclusion $i\colon d\A_t\to \widehat{\A}_t$. 
  We now claim that this has a right adjoint $\Gamma_t$ defined by sending $X=(\cE^{-1}R \otimes_S V \xrightarrow{\alpha} N)$ to the object 
  \[
  \Gamma_t(X)=(\cE^{-1}R \otimes_S V' \to (\cE^{-1}\ker(\alpha))/\ker(\alpha))
  \]
  as described by the following diagram 
  
  \[
  \begin{tikzcd}
  && \textcolor{red}{V'} \arrow[dr] \arrow[dl,red] & \\
  & \textcolor{red}{\cE^{-1} R \otimes_S V'} \ar[drrr,red] \arrow[dr] & & \textcolor{cyan}{V} \\
  \ker(\alpha) \arrow[rr] \arrow[dr,equal]& & \cE^{-1} \ker(\alpha) \arrow[rr] \arrow[uu, leftarrow, crossing over]  \arrow[dr]&&  \textcolor{red}{\cE^{-1}\ker(\alpha)/\ker(\alpha)} \arrow[dr]  \\
  &   \ker(\alpha) \arrow[rr]  && \textcolor{cyan}{\cE^{-1}R \otimes_S V} \arrow[rr, "\alpha"',cyan]  \ar[uu,crossing over, leftarrow, cyan]&& \textcolor{cyan}{N}
  \end{tikzcd}
  \]
  We move from the blue diagram to the red diagram by first taking $\ker(\alpha)$.  
  We note that the map $\ker(\alpha) \to \cE^{-1}R \otimes_S V$ must factor through $\cE^{-1}\ker(\alpha)$ since $\cE$ is invertible in $\cE^{-1}R \otimes_S V$. The cokernel of the map $\ker(\alpha)\to \cE^{-1}\ker(\alpha)$ gives the $\cE$-torsion object replacing $N$. We then construct the object $V'$ by taking the pullback of 
  $\cE^{-1} \ker(\alpha) \to \cE^{-1} R\otimes_S V \leftarrow V$ and note that we have a map from the extension of $V'$ via composition. 
  
  If $N$ is already $\cE$-torsion, then $\Gamma_t(X)=X$. Therefore the 
  unit of the adjunction $X \to \Gamma_t i X$ is given by the identity, whereas the 
  counit $i \Gamma_t Y \to Y$ is the map depicted in the previous diagram. 
  As such $(i, \Gamma_t)$ forms an adjoint pair.
  
\begin{lem}\label{lem:cells}
Let $\L = \{(\cE^{-1}R \to 0), (0 \to K)_{K \in \K}\}$ where $\K$ is a set of compact generators for the derived category of $\cE$-torsion $R$-modules. A bifibrant object $(\alpha\colon \cE^{-1}R \otimes_S V \to N) \in \widehat{\A}_t$ is bifibrant in the $\L$-cellularization if and only if the canonical $R$-module map $\Gamma_{\cE} N \to N$ is a quasi-isomorphism.
\end{lem}
\begin{proof}
We argue that the cells $(\cE^{-1}R \to 0)$ and $(0 \to K)$ 
are compact in $\widehat{\A}_t$ so we can identify the homotopy category of the cellularization with the localizing subcategory generated by the cells. The desired result then follows by a standard localizing subcategory argument.

It is clear by assumption that the cell 
$(0 \to K)$ is compact in $\widehat{\A}_t$, so it remains to show that $(\cE^{-1}R \to 0)$ is compact. Consider a map from $(\cE^{-1}R \to 0)$ into a direct 
sum $\bigoplus_i (\cE^{-1}R \otimes_S V_i \xrightarrow{\alpha_i} N_i)$. Note that any such map is uniquely determined by a map of $S$-modules $S \to \bigoplus_i \ker(\alpha)$ which factors through a finite stage since $S$ is a compact $S$-module. It follows that the cell $(\cE^{-1}R \to 0)$ is compact and this concludes the proof.
\end{proof}

We now state and prove the cellular skeleton theorem.

\begin{thm}\label{thm:cellularskeleton}
	Let $\L = \{(\cE^{-1}R \to 0), (0 \to K)_{K \in \K}\}$ where $\K$ is a set of compact generators for the derived category of $\cE$-torsion $R$-modules. The adjunctions 
	\[\begin{tikzcd}
	\cell{\L}\widehat{\A}_t^{\mathtt{dimp}} \arrow[r, yshift=1mm, "1"] & \cell{\L}\widehat{\A}_t^{\mathtt{dimi}}\arrow[l, yshift=-1mm, "1"] \arrow[r, yshift=-1mm, "\Gamma_t"'] & d\A_t^{\mathtt{dimi}} \arrow[l, yshift=1mm, "i"']
	\end{tikzcd} \]
	 are both Quillen equivalences.
\end{thm}
\begin{proof}
	For the first adjunction, without the cellularization the result is clear since the identity is a left Quillen equivalence from the projective to the injective model structure on modules. The cellularization preserves the Quillen equivalence.
	
	For the second adjunction, we firstly verify that it is a Quillen adjunction. Note that $$\Gamma_t: \widehat{\A}_t \leftrightarrows d\A_t : i$$ is a Quillen adjunction when both sides are equipped with the \texttt{dimi} model structure. Therefore, by the dual result to Dugger~\cite[A.2]{Dugger01}, to check that $$\Gamma_t: \cell{\L}\widehat{\A}_t^\texttt{dimi}\leftrightarrows d\A_t^\texttt{dimi} : i$$ is a Quillen adjunction it is sufficient to check that $i$ sends objects of $d\A_t$ to cellular objects. This is clear by Lemma~\ref{lem:cells}. To see that this Quillen adjunction is moreover a Quillen equivalence, firstly note that the unit is always an equivalence. The counit is an equivalence on objects of the form identified in Lemma~\ref{lem:cells} so the result follows.
\end{proof}

\subsection{Application to regular local rings of dimension 1}
We now apply the cellular skeleton theorem to simplify the torsion model for the derived category of a regular local ring of Krull dimension 1. Suppose that $R$ is a regular local ring which has Krull dimension 1 (i.e., $R$ is a discrete valuation ring), and write $\m$ for its maximal ideal. Take the family of prime ideals which consists of the maximal algebraic prime $\m$ (i.e., the minimal Balmer prime). Since $R$ is regular and has Krull dimension 1, the maximal ideal $\m$ is generated by a single element, say $x$. We can now identify the functors at play in the torsion model as $L\1 = R[x^{-1}]$, $\Gamma\1 = \Sigma^{-1}R/x^{\infty}$ and $\Lambda\1 = R_\m^\wedge$, see Section~\ref{sec:derived commutative algebra} and Remark~\ref{rem:regularlocal} for more details. Therefore the torsion model (Theorem~\ref{thm:torsionmodel}) gives a Quillen 
   equivalence
  \[
  \mod{R} \simeq _Q \mathrm{cell}\text{-} 
  \left(
  \begin{tikzcd}
   \mod{R[x^{-1}]} \arrow[d, "\bullet" description] & \\
   \mod{R_\m^\wedge[x^{-1}]} \arrow[r]
                               & \mod{R_\m^{\wedge}} 
  \end{tikzcd}
  \right)
  \]
  where the cellularization can be taken at the set of cells $\{(0 \to R/x),(R^\wedge_\m[x^{-1}] \to 0)\}$
 by Theorem~\ref{thm:bifibrant}. 
  
Define $\A_t(R)$ to be the category whose objects are $R_\m^\wedge$-module maps $R_\m^\wedge[x^{-1}] \otimes_{R[x^{-1}]} V \to T$ where $V$ is a $R[x^{-1}]$-module and $T$ is an $\m$-power torsion module. By Lemma~\ref{lem:abelian} this is an abelian category with enough injectives. The object $R/x$ is a compact generator for the derived category of $\m$-power torsion modules, and therefore applying the cellular skeleton theorem (Theorem~\ref{thm:cellularskeleton}) shows that we have a zig-zag of Quillen equivalences
 \[\mod{R} \simeq_Q d\A_t(R).\] 
 
 \begin{rem}\label{rem:regularlocal}
We restrict to the case of regular local rings of dimension 1 so that
$L\1$ is a ring, rather than a DGA. In general, for a local ring
$(R,\m)$ the localization $L\1 =
\check{C}_\m(R)$ can be additively obtained from $K_\infty(\m)$ by deleting $R$ from degree 0 and regrading. Therefore, if $\m=(x)$ is principal then $K_{\infty}(x)=(R\to R[x^{-1}])$ and $\check{C}_\m(R)=R[x^{-1}]$.
 \end{rem}

\section{The torsion model for rational $\bbT$-spectra}\label{sec:Tspectra}
  We now apply the theory developed in the previous sections to 
  establish an algebraic torsion model for the category of rational 
  $\bbT$-spectra: Theorem~\ref{thm-torsion model for T-spectra} promotes 
  the equivalence of \cite{Greenlees99} to a Quillen 
  equivalence.

  In this part of the paper we will be working rationally, i.e., all spectra 
  are rationalized without further comment, and all homology and cohomology 
  theories will be unreduced with rational coefficients. 
  \subsection{Preliminaries}
  Unless stated otherwise, we write $\bbT$ for the circle group and 
  $\Sp_{\mathbb{T}}$ for the category of orthogonal $\mathbb{T}$-spectra 
  with the rational $\mathbb{T}$-stable model structure. 
The Balmer spectrum was discussed in Section~\ref{sec:rational equivariant spectra}.

   We let $z$ denote the natural representation of $\bbT$, so that
   $z^n$ is a one dimensional representation with kernel $F$ cyclic of
   order $n$. The cohomology ring $H^*(B\mathbb{T}/F)=\Q [c_F]$ is
   polynomial on the Euler class $c_F$ of $z^n$.
  
\begin{defn}\leavevmode
  \begin{itemize} 
   \item  The family of finite subgroups of $\bbT$ will be denoted $\F$. 
  \item We denote by $\cO_{\F}$ the coefficient ring
  \[
  \cO_{\F}= \prod_{F \in \F} H^*(B\mathbb{T}/F).
  \]
  \item For a $\mathbb{T}$-representation $V$, the Euler class 
  $e(V) \in \cO_{\F}$ is the element with $F$-component 
  $ e(V)_F= e(V^{F}) \in H^{|V^{F}|}(B\mathbb{T}/F)$,
  and we write $\cE_\mathbb{T} = \{ e(V) \mid V^\mathbb{T}=0 \}$. 
  \item We define the $\F$-Tate ring by $t_{\F}= \cE_\mathbb{T}^{-1}\cO_{\F}$.   
  \end{itemize}
\end{defn}

  Recall from Section~\ref{sec:rational equivariant spectra} that 
  we have functors $L_{\F^c}, \Gamma_{\F}$ and $\Lambda_\F$ 
  associated to the family $\F$. Unravelling the definitions we see that 
  the diagram $T_\F^{\lrcorner}$ from Definition~\ref{defn:T-diagram}
  corresponds to the Tate diagram
 \[
 \begin{tikzcd}
                            & \widetilde{E}\F \arrow[d] \\
 DE\F_+ \arrow[r] & DE\F_+ \wedge \widetilde{E}\F.
 \end{tikzcd}
 \]
 
\begin{prop}\label{prop-formality}
By applying the $\mathbb{T}$-homotopy groups functor to the diagram $T_{\F}^{\lrcorner}$ above we obtain the following diagram of commutative rings
\[
 T^a_{\F}=
 \begin{tikzcd}
                            & \mathbb{Q} \arrow[d] \\
\cO_\F \arrow[r ] & t_\F
 \end{tikzcd}
 \]
 with the obvious maps. 
 Furthermore, the diagram $T^a_{\F}$ is intrinsically formal in the sense that 
 any cospan of commutative DGAs $A \to B \leftarrow C$ with homology $T^a_\F$
is quasi-isomorphic to $T^a_{\F}$.
\end{prop}
\begin{proof}
 The $\mathbb{T}$-homotopy groups of $DE\F_+$ can be calculated 
 using~\cite[2.2.3, 2.4.1]{Greenlees99}. It follows from~\cite[\S 5.2]{Greenlees99} that $\pi_*^{\mathbb{T}}(DE\F_+ \wedge \widetilde{E}\F)=t_\F$. Finally formality follows from~\cite[9.4, \S 9.D]{GreenleesShipley18}.
\end{proof}

\subsection{Euler torsion modules} 
  We would like to apply the theory of Section~\ref{sec:cellular skeleton thm} 
  to the diagram of rings $\cO_\F\to t_\F \leftarrow \mathbb{Q}$ equipped with 
  the torsion functor associated to the set $\cE_\bbT$ of Euler 
  classes. The Euler torsion functor is the composite of two functors which 
  both preserve injective objects by ~\cite[17.3.3, 17.3.5]{Greenlees99}. 
  Therefore by the results in Section~\ref{sec:cellular skeleton thm} there 
  exists an abelian category $\A_t(\mathbb{T})$ of injective dimension 2, 
  whose objects  have the form 
  $(t_\F \otimes V\to T)$ where $T$ is an Euler torsion $\cO_\F$-module and 
  $V$ is a $\mathbb{Q}$-module. In particular, the category 
  $d\A_t(\mathbb{T})$ of differential objects in $\A_t(\mathbb{T})$ 
  supports a $\texttt{dimi}$ model structure.
  
  In order to apply the cellular skeleton theorem we need to understand 
  the compact generators in the derived category of Euler torsion modules.   
  Let $e_F \in \cO_{\F}$ be the idempotent given by projection 
  onto the $F$-factor. Then we have the following useful criterion.

\begin{lem}[{\cite[4.6.6]{Greenlees99}}]\label{lem-euler torsion}
  The following conditions on a $\cO_{\F}$-module $M$ are equivalent:
  \begin{itemize}
  \item[(a)] $M$ is Euler torsion;
  \item[(b)] $M= \bigoplus_{F \in \F} e_F M$ and for each $m \in M$ there exist 
  $N>0$ such that $c^N m=0$, where $c$ is the total Chern class.
  \end{itemize}
\end{lem}
   
\begin{lem}\label{lem-torsion-minima}
 The derived 
 category of $x$-power torsion $\Q[x]$-modules is generated by any non-zero module. 
 In particular, it is compactly generated by $\Q$. 
\end{lem} 


\begin{cor}\label{cor-euler torsion}
  The derived category of Euler torsion $\cO_{\F}$-modules is compactly generated 
  by the modules $H^*(B\bbT/F)/(c_F)=\mathbb{Q}_F$ for finite subgroups $F$. 
\end{cor} 

\begin{proof}
  Let $\alpha_F \colon \mathbb{T} \to \mathbb{T}$ be a representation with kernel 
  exactly $F$, and consider $e(\alpha_F) \in H^2(B\mathbb{T}/F)$. 
  Lemma~\ref{lem-euler torsion} tells us that an $\cO_{\F}$-module $M$ is Euler 
  torsion if and only if it decomposes as $M=\bigoplus_{F}e_F M$ with 
  $e_F M$ an $e(\alpha_F)$-torsion $H^*(B\mathbb{T}/F)$-module. 
  It follows that the functor $M \mapsto (e_F M)_F$ defines an equivalence between 
  the category of Euler torsion $\cO_\F$-modules and the category 
  \[
  \coprod_{F \in \F} e(\alpha_F)\text{-tors-}H^*(B\mathbb{T}/F)\text{-mod}.
  \]
  Therefore the claim follows from Lemma~\ref{lem-torsion-minima} by recalling 
  that there is a canonical isomorphism of rings 
  $\mathbb{Q}[c_F] \to H^*(B\mathbb{T}/F) $ sending $c_F$ to $e(\alpha_F)$. 
\end{proof}

\subsection{Fixed points and commutativity}\label{sec:fiddlingwithmodels}
In this section we describe how to ensure a good interplay between commutativity and cofibrancy when taking categorical fixed points. This is harder than it first appears, due to the fact that categorical fixed points is not a Quillen functor between the flat model structures.

Suppose that we have a map $S \to R$ of commutative ring $G$-spectra. In addition, suppose that $R$ is a cofibrant $S$-module. Note that this can be forced if we work in the flat model structure, see Section~\ref{sec:convenient model structure} for more details. We wish to choose a suitable cofibrant replacement for $R^G$, denoted by $Q_f(R^G)$, with the following properties:
\begin{itemize}
\item[(a)] there is a map of commutative ring spectra $\phi\colon S^G \to Q_f(R^G)$;
\item[(b)] $\phi$ exhibits $Q_f(R^G)$ as a cofibrant $S^G$-module.
\end{itemize}
Condition (b) ensures that the restriction of scalars functor $\phi^*$ is left Quillen by Lemma~\ref{lem:restrictionleftQuillen}. This in turn guarantees that the sections of the required diagram category has a \texttt{dimp} model structure. Condition (a) is enforced to ensure that after applying Shipley's theorem~\cite[1.2]{Shipley07} we obtain commutative DGAs, and as such, may apply formality arguments.

We now show that a cofibrant replacement satisfying conditions (a) and (b) does exist. In addition, the replacement we construct will have other convenient properties which we state as Lemmas~\ref{lem:fixedpoints} and~\ref{lem:stabletoflat}. 

Taking categorical $G$-fixed points, we obtain a map of commutative ring spectra $S^G \to R^G$ since $(-)^G$ is lax symmetric monoidal. We now replace $R^G$ as a commutative $S^G$-algebra in the flat model structure, which gives a commutative diagram 
\[\begin{tikzcd}
S^G \arrow[r, hookrightarrow] \arrow[dr] & Q_f(R^G) \arrow[d, twoheadrightarrow, "\sim"] \\
& R^G
\end{tikzcd} \]
of commutative ring spectra. Then condition (a) is satisfied and (b) follows from~\cite[4.1]{Shipley04}. 

We note that one could also cofibrantly replace $R^G$ as an $S^G$-algebra in the stable model structure to give $\beta\colon Q(R^G) \to R^G$, but that it may not be possible to satisfy both (a) and (b) with such a replacement. Nonetheless, this alternative cofibrant replacement acts as a useful stepping stone between $G$-spectra and non-equivariant spectra. This will play an important role in the next section, via the following two lemmas.
\begin{lem}\label{lem:fixedpoints}
There is a composite Quillen adjunction
\[\begin{tikzcd}\mathrm{Mod}_R^\mathrm{flat}(\Sp_G) \arrow[r, "1"', yshift=-1mm] & \mathrm{Mod}_R^\mathrm{stable}(\Sp_G) \arrow[l, "1"', yshift=1mm] \arrow[rrr, yshift=-1mm, "(-)^G"'] & & & \mathrm{Mod}_{R^G}^\mathrm{stable}(\Sp) \arrow[lll, yshift=1mm, "R \otimes_{\mathrm{inf}R^G} \mathrm{inf}(-)"'] \arrow[rr, yshift=-1mm, "\beta^*"'] & & \mathrm{Mod}_{Q(R^G)}^\mathrm{stable}(\Sp). \arrow[ll, yshift=1mm, "\beta_*"'] \end{tikzcd}\] 
Moreover, this composite Quillen adjunction is a Quillen equivalence if and only if the middle adjunction is a Quillen equivalence.
\end{lem}
\begin{proof}
The middle adjunction arises from the adjoint lifting theorem since $(-)^G$ is lax monoidal; see~\cite[\S 3.3]{SchwedeShipley03} or~\cite{GreenleesShipley14c} for more details. The left hand adjunction is a Quillen equivalence since the weak equivalences in the stable and flat model structure are the same, and the right hand adjunction is a Quillen equivalence since $\beta$ is a weak equivalence.
\end{proof}

\begin{lem}\label{lem:stabletoflat}
There is a map of ring spectra $\alpha\colon Q(R^G) \to Q_f(R^G)$ which is a weak equivalence. Therefore there is a Quillen equivalence
\[\begin{tikzcd}
\mathrm{Mod}_{Q(R^G)}^{\mathrm{stable}} \arrow[r, yshift=1mm, "\alpha_*"] & \mathrm{Mod}_{Q_f(R^G)}^{\mathrm{flat}} \arrow[l, yshift=-1mm, "\alpha^*"]
\end{tikzcd} \]
given by extension and restriction of scalars. 
\end{lem}
\begin{proof}
There is a commutative square
\[\begin{tikzcd}
 S^G \arrow[r] \arrow[d, hookrightarrow, "\mathrm{stable}"'] & Q_f(R^G)\arrow[d, twoheadrightarrow, "\sim"', "\mathrm{flat}"] \\
 Q(R^G) \arrow[r, twoheadrightarrow, "\sim", "\mathrm{stable}"'] & R^G\end{tikzcd}\]
of ring spectra. Since the right hand vertical is an acyclic flat fibration and hence an acyclic stable fibration, there is a lift $Q(R^G) \to Q_f(R^G)$ in the square as a map of rings. Moreover, this is a weak equivalence by the 2-out-of-3 property. The claimed Quillen equivalence follows since the identity functor is a right Quillen equivalence from the flat model structure to the stable model structure.
\end{proof}

\subsection{The Quillen equivalences}
The goal of this section is to prove the following result.
  
\begin{thm}\label{thm-torsion model for T-spectra}
	There is a Quillen equivalence 
	\[
	\Sp_\mathbb{T} \simeq_Q d\A_t(\mathbb{T}).
	\]
\end{thm}

  We divide the proof into several steps. Since the formality argument is a vital step and it relies upon \emph{commutativity}; we must use a convenient model structure, see the discussion in Section~\ref{sec:convenient model structure}. Therefore, we will work with the flat model structure on equivariant spectra as in~\cite{BrunDundasStolz}. This introduces some complications since the categorical fixed points functor is not Quillen as a functor $\Sp_G^\mathrm{flat} \to \Sp^\mathrm{flat}$. We shall use the results of Section~\ref{sec:fiddlingwithmodels} to overcome these complications.
  
\subsection*{Step 1}
   Consider the category $\Sp_{\mathbb{T}}$ with $L=L_{\F^c}$ the
   $\F$-nullification functor given by
   $LX=\widetilde{E}\F \sm X$. 
   We apply Theorem~\ref{thm:torsionmodel} to deduce a Quillen 
   equivalence
  \[
  \Sp_\mathbb{T} \simeq _Q \mathrm{cell}\text{-} 
  \left(
  \begin{tikzcd}
   \mod{\widetilde{E}\F} \arrow[d,"\bullet" description] & \\ 
   \mod{DE\F_+ \wedge \widetilde{E}\F} \arrow[r]
                               & \mod{DE\F_+} 
  \end{tikzcd}
  \right)
  \]
  where the right hand side uses the flat model structure.

\subsection*{Step 2}
  The next step is to take categorical fixed points to remove equivariance. Since the categorical $\mathbb{T}$-fixed points functor  
  $\Sp_{\mathbb{T}} \to \Sp$ is \emph{not} right Quillen between the flat model 
  structures we must employ the strategy described in Section~\ref{sec:fiddlingwithmodels}. Loosely speaking we first neglect the commutativity to ensure that the relevant functors are Quillen. Once we have then passed to the non-equivariant world, one can reclaim commutativity of the rings.
  
  Firstly we recall that the following result.
  \begin{lem}\label{lem:EM}
 There are Quillen equivalences
\begin{itemize}
\item \cite[4.3, 9.1]{GreenleesShipley14c}
  \[
\mod{DE\F_+} \simeq_Q \mod{(DE\F_+)^\mathbb{T}}
  \]
\item \cite[7.1]{GreenleesShipley14c}
\label{lem:EFtildegenerates}
  \[
\mod{\widetilde{E}\F} \simeq_Q \mod{S^0}
  \]
\item \label{lem:DEF and EFtilde generate}
\cite[7.1]{GreenleesShipley14c}
  \[
  \mod{DE\F_+ \wedge \widetilde{E}\F} \simeq_Q
  \mod{(DE\F_+ \wedge \widetilde{E}\F)^{\mathbb{T}}}
  \]
\end{itemize}
\end{lem}
 
We now use the notation of Section~\ref{sec:fiddlingwithmodels} for the necessary cofibrant replacements. There is a Quillen equivalence 
\[\left(\begin{tikzcd}
\mod{\widetilde{E}\F} \arrow[d, "\bullet" description] \\
\mod{DE\F_+ \sm \widetilde{E}\F} \arrow[r] & \mod{DE\F_+}
\end{tikzcd}\right)\leftrightarrows 
\left(\begin{tikzcd}
\mod{S^0} \arrow[d, "\bullet" description] \\
\mod{(DE\F_+ \wedge \widetilde{E}\F)^{\mathbb{T}}} \arrow[r] & \mod{Q(DE\F_+^\mathbb{T})}
\end{tikzcd}\right)\]
where the left hand side is equipped with the flat model structure, and the right hand side is equipped with the stable model structure. This Quillen equivalence follows from Lemmas~\ref{lem:fixedpoints} and~\ref{lem:EM}.

Next, we claim that there is a Quillen equivalence
\[\left(\begin{tikzcd}
\mod{S^0} \arrow[d, "\bullet" description] \\
\mod{(DE\F_+ \wedge \widetilde{E}\F)^{\mathbb{T}}} \arrow[r] & \mod{Q(DE\F_+^\mathbb{T})}
\end{tikzcd}\right) \rightleftarrows \left(\begin{tikzcd}
\mod{S^0} \arrow[d, "\bullet" description] \\
\mod{(DE\F_+ \wedge \widetilde{E}\F)^{\mathbb{T}}} \arrow[r] & \mod{Q_f(DE\F_+^\mathbb{T})}
\end{tikzcd}\right) \]
where the left hand side has the stable model structure, but the right hand side has the flat model structure. The functors are the identity on the leftmost vertices and as such are Quillen equivalences since the identity functor from the stable to the flat model structure is a left Quillen equivalence. In the bottom right vertex, applying Lemma~\ref{lem:stabletoflat} shows that we have a Quillen equivalence. We emphasize that each of the rings appearing on the right hand side is commutative; this is vital for the next step.

\subsection*{Step 3}
  We now apply Shipley's theorem~\cite[1.2]{Shipley07} which says that given a commutative $H\mathbb{Q}$-algebra $A$, there is a commutative DGA $\Theta A$ and a Quillen equivalence $\mod{A} \simeq_Q \mod{\Theta A}$, and that moreover $H_*(\Theta A) = \pi_*A$. Therefore we have a Quillen equivalence
   \[
  \left(
  \begin{tikzcd}
   \mod{S^0} \arrow[d, "\bullet" description] & \\ \mod{(DE\F_+ \wedge \widetilde{E}\F)^\mathbb{T}} \arrow[r]
                               & \mod{Q_f(DE\F_+^\mathbb{T})} 
  \end{tikzcd}
  \right)
  \simeq_Q
  \left(
  \begin{tikzcd}
  \mod{\Theta (S^0)} \arrow[d,"\bullet" description] & \\ \mod{\Theta ((DE\F_+ \wedge \widetilde{E}\F)^\mathbb{T})} \arrow[r]
                     & \mod{\Theta (Q_f(DE\F_+^\mathbb{T}))}
  \end{tikzcd}
  \right)
  \]
  where each of the DGAs appearing is commutative. By Proposition~\ref{prop-formality}, we know that the cospan $\cO_\F \to t_\F \leftarrow \mathbb{Q}$ is intrinsically formal as a cospan of commutative DGAs.
 This gives a Quillen equivalence 
  \[
  \left(
  \begin{tikzcd}
  \mod{\Theta (S^0)} \arrow[d,"\bullet" description] & \\ \mod{\Theta ((DE\F_+ \wedge \widetilde{E}\F)^\mathbb{T})} \arrow[r]
                     & \mod{\Theta (Q_f(DE\F_+^\mathbb{T}))}
  \end{tikzcd}
  \right)
  \simeq_Q
  \left(
  \begin{tikzcd}
  \mod{\Q} \arrow[d,"\bullet" description] & \\ \mod{t_{\F}} \arrow[r]
                     & \mod{\cO_{\F}}
  \end{tikzcd}
  \right).
  \]
 We note that Shipley's algebraicization theorem still holds in the flat model structure by~\cite{Williamsonflat}.
  To simplify notation we call the (underlying) category on the right $\widehat{\A}_t(\mathbb{T})$. 
  
\subsection*{Step 4}
Combining the results of the previous steps, we see that we have a zig-zag of Quillen equivalences \[\Sp_\bbT \simeq_Q \mathrm{cell}\text{-}\widehat{\A}_t(\mathbb{T}).\] We now describe the effect of the cellularization on 
  $\widehat{\A}_t(\mathbb{T})$ and use the cellular skeleton theorem to internalize the cellularization.  
  Recall that an object in 
  $\widehat{\A}_{t}(\mathbb{T})$ consists of a $\mathbb{Q}$-module $V$, a 
  $t_{\F}$-module $W$ and an $\cO_\F$-module $N$, together with module maps
  $\alpha \colon t _{\F}\otimes V \to W$ and $\beta \colon W \to N$. 
  We also required the $t_{\F}$-module map $\alpha$ to be an isomorphism. 
  As this data is equivalent to an $\cO_\F$-module map 
  $t_{\F} \otimes V \to N$, we will use this shorthand 
  notation to denote an object in the category $\widehat{\A}_t(\mathbb{T})$.
  
\begin{lem}\label{lem-cell}
 The cellularization of $\widehat{\A}_t(\mathbb{T})$ can be taken at the set of cells \[\{(t_\F \to 0), (0 \to \mathbb{Q}_F)_{F \in \F}\}.\]
\end{lem}
\begin{proof}
We first determine the homotopy of the cells. For $H \in \F$, we put $\mathbb{I}(H)= \Sigma^{-2} \mathbb{Q}[c_H, c_H^{-1}]/ \mathbb{Q}[c_H]$ which is a torsion injective $H^*(B\mathbb{T}/H)$-module, and write $A(H)$ for the $H$-Burnside ring which is an Euler torsion $\cO_{\F}$-module. We use~\cite[2.4.3]{Greenlees99} to see that $\pi_*^{\mathbb{T}}(\Sigma E\F_+)= \bigoplus_{F\in \F} \Sigma^2 \mathbb{I}(F)$. The isotropy separation sequence and the tom Dieck splitting show that
$\pi_*^{\mathbb{T}}(\Sigma E \F_+ \wedge \mathbb{T}/H_+)= \pi_*^{\mathbb{T}}(\Sigma \mathbb{T}/H_+)=\Sigma^2 A(H)$. It follows that the cells have homotopy
\[
\pi_*^{\mathbb{T}}((\mathbb{T}/\mathbb{T}_+)_\mathrm{tors})= (t_{\F}\to \bigoplus_{F \in \F} \Sigma^2 \mathbb{I}(F)) \quad \mathrm{and} \quad 
\pi_*^{\mathbb{T}}((\mathbb{T}/H_+)_\mathrm{tors})=(0 \to \Sigma^2 A(H))
\] for $H \in \F$.
Using Corollary~\ref{cor-euler torsion}, one can see that the localizing subcategory generated by the cells $\{\pi_*^{\mathbb{T}}((\mathbb{T}/H_+)_\mathrm{tors}) \mid H \leq \mathbb{T} \}$ is the same as that generated by $( t_{\F} \to 0)$ and 
$(0 \to \mathbb{Q}_F)_{F \in \F}$ so they have the same cellularization.
\end{proof}

Since $\{(0 \to \mathbb{Q}_F)\}_{F \in \F}$ is a set of compact generators for Euler torsion modules by Corollary~\ref{cor-euler torsion}, applying the cellular skeleton theorem (Theorem~\ref{thm:cellularskeleton}) then gives a Quillen equivalence
\[\text{cell-}\widehat{\A}_t(\bbT) \simeq_Q d\A_t(\bbT) \]
which completes the proof of Theorem~\ref{thm-torsion model for T-spectra}.

\bibliographystyle{plain}
\bibliography{torsion}
 \end{document}